\documentclass[10pt,oneside,leqno]{amsart}
\usepackage{amsxtra}
\usepackage{amsopn}
\usepackage{color}
\usepackage{amsmath,amsthm,amssymb}
\usepackage{amscd}
\usepackage{amsfonts}
\usepackage{latexsym}
\usepackage{verbatim}
\usepackage{graphicx}
\usepackage{graphics}
\usepackage{multirow}
\usepackage{lscape}
\usepackage{scrextend}
\usepackage{mathabx}

\usepackage[hypertexnames=false,
 backref=UseNone,
    pdftex,
    pdfpagemode=UseNone,
    breaklinks=true,
    extension=pdf,
    colorlinks=true,
    linkcolor=blue,
    citecolor=blue,
    urlcolor=blue,
]
{hyperref}

\theoremstyle{plain}
\newtheorem{theorem}{Theorem}[section]
\newtheorem{definition}[theorem]{Definition}
\newtheorem{lemma}[theorem]{Lemma}
\newtheorem{proposition}[theorem]{Proposition}
\newtheorem{corollary}[theorem]{Corollary}
\newtheorem{remark}[theorem]{Remark}
\newtheorem{notation}[theorem]{Notation}
\newtheorem{example}[theorem]{Example}

\newtheorem{remark-question}[section]{Remark-Question}

\newcommand\fra{{\mathfrak a}} 

\newcommand\frg{{\mathfrak g}}
\newcommand\frh{{\mathfrak h}}

\newcommand\Real{{\mathfrak R}{\frak e}\,} 

\definecolor{fondo}{rgb}{0.93,0.93,0.93}

\definecolor{m}{rgb}{0.9,0,0.9}

\makeatletter
\renewcommand*{\eqref}[1]{%
  \hyperref[{#1}]{\textup{\tagform@{\ref*{#1}}}}%
}
\makeatother

\begin{document}
\title[]{The ascending central series of nilpotent Lie algebras\\ with complex structure}

\author{Adela Latorre}
\address[A. Latorre and R. Villacampa]{Centro Universitario de la Defensa\,-\,I.U.M.A., Academia General
Mili\-tar, Crta. de Huesca s/n. 50090 Zaragoza, Spain}
\email{adela@unizar.es}
\email{raquelvg@unizar.es}

\author{Luis Ugarte}
\address[L. Ugarte]{Departamento de Matem\'aticas\,-\,I.U.M.A.\\
Universidad de Zaragoza\\
Campus Plaza San Francisco\\
50009 Zaragoza, Spain}
\email{ugarte@unizar.es}

\author{Raquel Villacampa}


\begin{abstract}
We obtain several restrictions on the terms of the ascending central series of a nilpotent Lie algebra $\frg$ under the presence of a complex structure $J$.  In particular, we find a bound for the dimension of the center of $\frg$ when it does not contain any non-trivial $J$-invariant ideal.
Thanks to these results, we provide a structural theorem
describing the ascending central series of 8-dimensional nilpotent Lie algebras $\frg$ admitting this particular type of complex structures $J$.
Since our method is constructive, it allows us to describe the complex structure equations
that parametrize all such pairs $(\frg, J)$.
\end{abstract}

\maketitle

\setcounter{tocdepth}{3} \tableofcontents


\section{Introduction}

\noindent
Let $\frg$ be an even-dimensional real Lie algebra.  A complex structure on $\frg$ is an endomorphism $J\colon\frg\longrightarrow\frg$
satisfying $J^2=-\textrm{Id}$ and the ``Nijenhuis condition''
\begin{equation}\label{Nijenhuis}
N_J(X,Y):=[X,Y]+J[JX,Y]+J[X,JY]-[JX,JY]=0,
\end{equation}
for all $ X,Y\in\frg$.
Finding Lie algebras endowed with such structures constitutes an interesting problem with important algebraic and geometrical applications.
In the last years several results on classifications of Lie algebras $\frg$ admitting complex structures $J$ have been published,
mainly dealing with low dimensions.
In the $4$-dimensional case, the solvable Lie algebras admitting a complex structure were classified by Ovando in~\cite{O}. Concerning dimension $6$, Andrada, Barberis, and Dotti classified in~\cite{ABD} the pairs $(\frg,J)$
where $\frg$ is any Lie algebra and $J$ is an abelian complex structure,
i.e. $J$ satisfies $[JX, JY] = [X,Y]$ for every $X,Y\in\frg$.
The classification of $6$-dimensional nilpotent Lie algebras that admit a complex structure (not necessarily abelian) was achieved by Salamon in~\cite{S}. Later, the different complex structures on each of these algebras were classified by Ceballos, Otal, Ugarte, and Villacampa in~\cite{COUV}.
In the 6-dimensional solvable case, complex structures of certain specific types are studied in~\cite{AOUV} and~\cite{FOU}.
The existence of complex structures on $6$-dimensional product Lie algebras have been recently
investigated in \cite{CS}.
However, apart from some classifications of complex paralellizable structures,
i.e. those coming from complex Lie algebras (see e.g. \cite{G,Nakamura}),
little is yet known in real dimensions higher than six.

In this paper we focus on nilpotent Lie algebras.
Fixed an even real dimension $2n$,
it seems clear that
the knowledge of a classification of $2n$-dimensional nilpotent Lie algebras could
be of great help in the study of existence of complex structures.
Nonetheless, complete classifications are only known up to dimension~7 (see~\cite{G} and the references therein).
Hence, in order to
investigate the pairs $(\frg, J)$ in dimension~$2n\geq 8$, other techniques are needed.
In fact, when dealing with high dimensional cases, most of the efforts have been directed to obtain
algebraic constraints to the existence of complex structures, see for instance the
papers \cite{Cav-Gualt,GVR,GR} for quasi-filiform Lie algebras,
and the results in~\cite{Mi} bounding the nilpotency step of nilpotent Lie algebras admitting complex structures.
Although some partial classifications are
achieved in eight dimensions, they generally require additional conditions,
such as the existence of hypercomplex structures~\cite{DF},
SKT metrics~\cite{EFV}, or balanced metrics compatible with abelian complex structures~\cite{AV}.
Some other partial results on the construction of complex structures have been recently obtained in~\cite{CCO}.

Our main objective here is to develop a method for constructing
all complex structures on even-dimensional nilpotent Lie algebras $\frg$.
This method will allow us to obtain several restrictions
on the terms of the ascending central series of $\frg$ imposed by the existence of a complex structure.  Among them, we prove an upper bound on the dimension of the center of $\frg$ when this subspace does not contain any non-trivial $J$-invariant ideal.  In such case, a structural theorem in eight dimensions is provided, together with a parametrization of
the space of complex structures.
We next explain in detail the contents of the paper.

\smallskip

In Section~\ref{complex-structures} we consider a partition of the space of complex structures $J$ on a
nilpotent Lie algebra $\frg$ into
\emph{quasi-nilpotent} and \emph{strongly non-nilpotent} structures (see Definition~\ref{tipos_J}
and Figure~\ref{diagrama_tipos_J}).
The first class is given by those~$J$'s
for which there exists a non-trivial $J$-invariant ideal in the center of~$\frg$.  We notice that this class contains those
complex structures of nilpotent type~\cite{CFGU-dolbeault}. The second class
appeared for the first time in~\cite{CFGU-proceeding}, and we will simply refer to
it as the class of
\emph{SnN} complex structures.
After showing that quasi-nilpotent complex structures can be constructed as a certain extension of lower dimensional structures,
we notice that the essentially new complex structures
that arise in each even real dimension
are the SnN ones.
For this reason, in the rest of the paper we mainly focus on these complex structures.

In Section~\ref{SnN} we obtain algebraic constraints to the existence of strongly non-nilpotent complex
structures $J$ on a nilpotent Lie algebra $\frg$ in terms of its ascending central series $\{\frg_k\}_k$.
One can think of SnN complex structures as those $J$'s for which the spaces $\frg_k$ are
far from being $J$-invariant.
This will be reflected in the fact that the center $\frg_1$ of $\frg$ is rather small, whereas the nilpotency step of $\frg$
can be rather large.
Indeed,
in Section~\ref{sec:restrictions} we prove, among other restrictions for the terms $\frg_k$, that
the nilpotency step $s$ of $\frg$ satisfies $s\geq 3$,
and in Section~\ref{proof} we obtain an upper bound for the dimension of $\frg_1$.
An application to the study of existence of complex structures
on products of nilpotent Lie algebras is given in~Section~\ref{aplications}.

Section~\ref{sec:8-dim} is devoted to dimension 8.
The main result is Theorem~\ref{teorema-estructura-acs}, which gives the structure of the ascending central series of 8-dimensional
nilpotent Lie algebras $\frg$ admitting
an SnN complex structure~$J$. For the proof of this result we construct
a \emph{doubly adapted} basis
for each pair $(\frg,J)$ (see Definition~\ref{def-doubly-adapted}) which allows us to explicitly describe the terms $\frg_k$ in the ascending central series of~$\frg$.
Indeed, in eight dimensions one has that dim\,$\frg_1=1$ and $\frg_5=\frg$ (since $s\leq 5$).
In Section~\ref{sec:g2} we focus on the space~$\frg_2$ and find its possible dimensions, together with
its description in terms of appropriate generators. A similar analysis is made in
Section~\ref{sec:g34} for $\frg_3$ and $\frg_4$.
Finally, since our method is constructive, it allows us to describe in Section~\ref{SnN-equations-dim8}
the complex structure equations
that parametrize all the strongly non-nilpotent complex structures on nilpotent Lie algebras of dimension~8.


\section{Classes of complex structures on a nilpotent Lie algebra $\frg$}\label{complex-structures}

\noindent
In this section we introduce different classes of complex structures on a nilpotent Lie algebra $\frg$, and
we show that those complex structures $J$ for which the center of $\frg$ contains a non-trivial $J$-invariant ideal
are a certain extension of lower dimensional pairs $(\tilde \frg, \tilde J)$.
As a consequence, we arrive at the fact that the essentially new complex structures that arise in each even real dimension
are the so-called strongly non-nilpotent complex structures.

\smallskip
The \emph{ascending central series} of a Lie algebra $\frg$ is given by $\{\frg_k\}_{k}$, where
\begin{equation}\label{ascending-series}
\left\{\begin{array}{l}
\frg_0=\{0\}, \text{ and } \\[4pt]
\frg_k=\{X\in\frg \mid [X,\frg]\subseteq \frg_{k-1}\}, \text{ for } k\geq 1.
\end{array}\right.
\end{equation}
Observe that $\frg_1=Z(\frg)$ is the center of~$\frg$.
The Lie algebra $\frg$ is said to be \emph{nilpotent} if there is
an integer $s\geq 1$ such that $\frg_k=\frg$ for every~$k\geq s$.
In this case, the smallest integer $s$ satisfying the latter condition is called the
\emph{nilpotency step} of~$\frg$, and the Lie algebra is said to be \emph{$s$-step nilpotent}.

Let $\frg$ be a nilpotent Lie algebra (NLA for short) of real dimension $2n$. We introduce the following terminology:

\begin{definition}\label{ascending-type}
An NLA $\frg$ {\rm has ascending type $(m_{1},\ldots,m_{s})$}
if $s$ is the nilpotency step of~$\frg$ and~$m_k=\dim \frg_{k}$ for each~$1\leq k\leq s$, being $\frg_k$ the terms in the
ascending central series $\{\frg_k\}_{k}$ of $\frg$.
\end{definition}

Thus, any NLA $\frg$ has an associated $s$-tuple
$$(m_{1},\ldots,m_{s-1},m_s) :=\left(\text{dim\,}\frg_{1},\ldots, \text{dim\,}\frg_{s-1},\text{dim\,}\frg_{s}\right)$$
which strictly increases, i.e. $0<m_{1}<\cdots <m_{s-1}<m_{s}=2n$.
Notice that $m_{1}$ is precisely the dimension of the center of $\frg$. Moreover, since $\frg$ is nilpotent one has that $\text{dim}\,\frg_{s-1}\leq2(n-1)$, and the nilpotency step verifies $s\leq 2n-1$.

Obviously, NLAs with different ascending type are non-isomorphic. However, the converse is only true up to real dimension 4.
Indeed, there are three non-isomorphic 4-dimensional NLAs whose ascending types are $(4)$, $(2,4)$, and $(1,2,4)$.
In contrast, there exist four non-isomorphic 6-dimensional NLAs with the same ascending type $(2,6)$
(see for instance~\cite{CFGU-palermo}, where the ascending type of every 6-dimensional NLA is given).

\smallskip
We observe that the existence of a complex structure $J$ on an NLA $\frg$ imposes restrictions on the ascending type of $\frg$.
These restrictions are completely known up to dimension 6 (see \cite{CFGU-palermo,S,U}),
but they are still to be understood in higher dimensions.
With this goal in mind, we will first reduce the problem to a specific class of complex structures that we will next introduce.

\smallskip
Let $J$ be a complex structure on $\frg$, that is, an endomorphism $J\colon\frg\longrightarrow\frg$
satisfying $J^2=-\textrm{Id}$ and the integrability condition~\eqref{Nijenhuis}.
As it was observed in \cite{CFGU-dolbeault}, the terms $\frg_{k}$ in the series \eqref{ascending-series}
may not be invariant under $J$.
For this reason, a new series $\{\fra_{k}(J)\}_{k}$ adapted to the complex structure~$J$ is introduced in \cite{CFGU-dolbeault}, namely:
\begin{equation}\label{adapted-series}
\left\{\begin{array}{l}
\fra_0(J)=\{0\}, \text{ and } \\[4pt]
\fra_k(J)=\{X\in\frg \mid [X,\frg]\subseteq \fra_{k-1}(J)\ {\rm and\ } [JX,\frg]\subseteq \fra_{k-1}(J)\}, \text{ for } k\geq 1.
\end{array}\right.
\end{equation}
We will refer to the series $\{\fra_{k}(J)\}_{k}$ as the \emph{ascending $J$-compatible series of~$\frg$}.

Note that every $\fra_k(J)\subseteq\frg_{k}$ is an even-dimensional $J$-invariant ideal of $\frg$ satisfying
$0 \leq \dim\,\fra_k(J) \leq m_k$, for every $k \geq 1$.
In particular,
$\fra_1(J)$ is the largest subspace
of the center $\frg_1$ which is $J$-invariant.


\begin{definition}\label{tipos_J}
A complex structure $J$ on an NLA $\frg$ is said to be
\begin{itemize}
\item[(i)] \emph{strongly non-nilpotent}, or \emph{SnN} for short, if $\fra_1(J)=\{0\}$;

\smallskip
\item[(ii)] \emph{quasi-nilpotent}, if it satisfies $\fra_1(J)\neq\{0\}$; moreover, $J$ will be called
 \begin{itemize}
 \item[(ii.1)] \emph{nilpotent}, if there exists an integer $t>0$ such that $\fra_t(J)=\frg$,
 \item[(ii.2)] \emph{weakly non-nilpotent}, if there is an integer $t>0$ such that $\fra_t(J)=\fra_l(J)$, for every $l\geq t$,
           and $\fra_t(J)\neq\frg$.
 \end{itemize}
\end{itemize}
\end{definition}

The notion of nilpotent complex structure was first introduced and studied in~\cite{CFGU-dolbeault}, whereas that of
strongly non-nilpotent appeared for the first time in~\cite{CFGU-proceeding}.

Let us remark that the first division above is based on whether the ascending $J$-compatible series~$\{\fra_{k}(J)\}_{k}$ of~$\frg$
is trivially zero or not.
Also notice that non-nilpotent structures are those
satisfying~$\fra_k(J)\neq\frg$, for every $k\geq 1$,
and they can be either weakly or strongly non-nilpotent.

\smallskip
\begin{figure}[h]
\begin{center}
\includegraphics[scale=0.48]{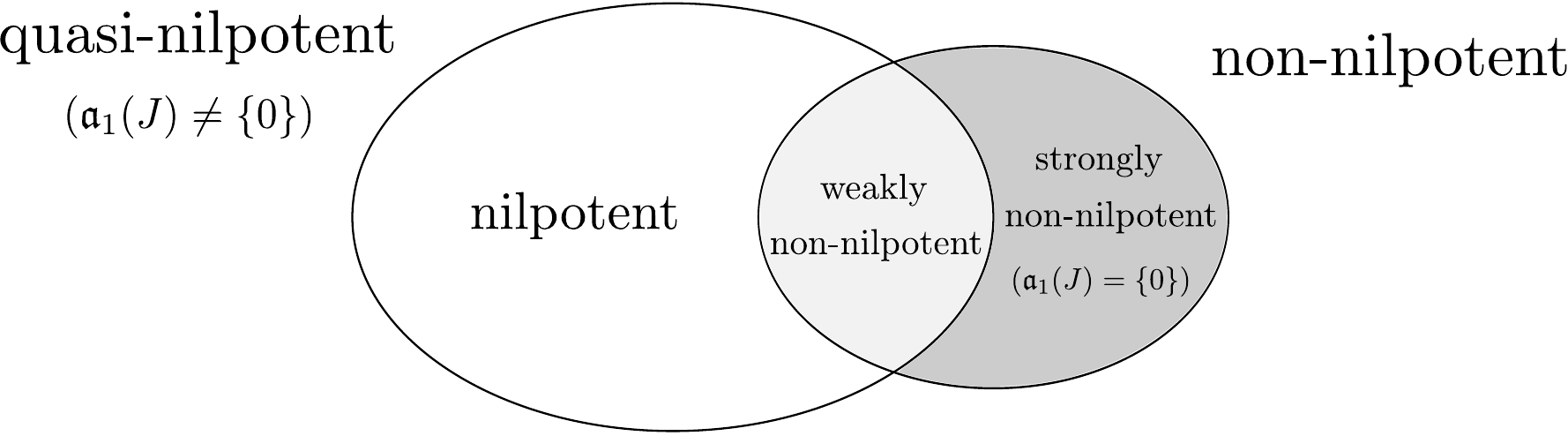}
\caption{Partition of the space of complex structures.}
\label{diagrama_tipos_J}
\end{center}
\end{figure}

It is well known that in dimension 4, any complex structure on an NLA is nilpotent.
In dimension~6,
any complex structure is either nilpotent or strongly non-nilpotent.
Moreover, these two classes cannot coexist on the same 6-dimensional NLA
(see \cite{CFGU-palermo,S,U}).

We next show that weakly non-nilpotent complex structures appear for the first time in dimension~8.
Furthermore, we illustrate that different classes of complex structures
can coexist on the same NLA when the dimension is greater than~6.


\begin{example}\label{ejemplo8}
Let $\frg$ be the 8-di\-men\-sion\-al NLA given by the basis $\{X_1,\ldots,X_8\}$, in terms of which the non-zero brackets are
$$[X_1,X_3]=[X_2,X_4]=X_6,\quad [X_3,X_5]=-X_1,\quad [X_4,X_5]=-X_2.$$
The ascending central series~\eqref{ascending-series} of $\frg$ is
the following one:
$$\frg_1=\langle X_6,X_7,X_8\rangle,\quad \frg_2=\langle X_1,X_2,X_6,X_7,X_8\rangle,\quad \frg_3=\frg,$$
which implies that $\frg$ has ascending type $(3,5,8)$.

Let $J$ be the almost complex structure defined by
$$JX_1=X_2,\ JX_3=X_4,\ JX_5=X_6,\ JX_7=X_8.$$
It is easy to see that $N_J \equiv 0$, i.e. $J$ is integrable, and the ascending $J$-compatible series~\eqref{adapted-series} of $\frg$ satisfies
$\mathfrak{a}_1(J)=\mathfrak{a}_l(J)=\langle X_7,X_8\rangle$, for every $l\geq 1$.
Hence, the complex structure $J$ is weakly non-nilpotent.

Let us observe that we can also define a nilpotent complex structure $\hat{J}$ on $\frg$ as follows:
$$\hat{J}X_1=X_2,\quad \hat{J}X_3=X_4,\quad \hat{J}X_5=X_8,\quad \hat{J}X_6=-X_7.$$
Indeed, $N_{\hat{J}} \equiv 0$ and $\fra_1(\hat J)=\langle X_6,X_7\rangle$,
$\fra_2(\hat J)=\langle X_1,X_2,X_6,X_7\rangle$, and $\fra_3(\hat J)=\frg$.

\end{example}

The following example shows that even nilpotent and strongly non-nilpotent complex structures can coexist on the same NLA.
It should be noted that the dimension of the Lie algebra is 10 (see Corollary~\ref{corollary-estructura-acs}, where this
is proved to be the smallest dimension where this phenomenon can occur).


\begin{example}\label{ejemplo10}{\rm \cite{CFGU-dolbeault}}
Let $\frg$ be the nilpotent Lie algebra of dimension 10 defined by a basis $\{X_1,\ldots,X_{10}\}$ with non-zero brackets
$$[X_3,X_9]=[X_4,X_{10}]=X_1,\quad [X_5,X_9]=[X_6,X_{10}]=X_2,$$
$$[X_7,X_9]=X_3,\quad [X_7,X_{10}]=X_4,\quad [X_8,X_9]=X_5,\quad [X_8,X_{10}]=X_6.$$
It is easy to see that $\frg_1=\langle X_1,X_2\rangle$, $\frg_2=\langle X_1,X_2,X_3,X_4,X_5,X_6\rangle$, and $\frg_3=\frg$.
Hence, $\frg$ has ascending type $(2,6,10)$.

The almost complex structure $J$ given by
$$JX_1=-X_7,\quad JX_2=-X_8,\quad JX_3=X_4,\quad JX_5=X_6,\quad JX_9=X_{10}$$
is integrable. Moreover, $J$ is
strongly non-nilpotent since  $\fra_1(J)=\{0\}$.

On the other hand, $\hat{J}$ defined by
$$\hat{J}X_1=X_2,\quad \hat{J}X_3=X_4,\quad \hat{J}X_5=X_6,\quad \hat{J}X_7=X_8,\quad \hat{J}X_9=X_{10}$$
is also a complex structure on $\frg$, and it satisfies
$\fra_1(\hat{J})=\frg_1$, $\fra_2(\hat{J})=\frg_2$, and $\fra_3(\hat{J})=\frg$.
Therefore, $\hat{J}$ is a nilpotent complex structure on $\frg$.
\end{example}

One can see from Definition~\ref{tipos_J}  that complex structures $J$ of quasi-nilpotent type on an NLA $\frg$ are,
up to a certain point, compatible with the nilpotent structure of $\frg$.
One can take advantage of this fact, as we next show following the ideas in~\cite{CFGU-proceeding,CFGU-dolbeault}.

Let $J$ be a complex structure on an NLA $\frg$ such that $\fra_1(J)\neq\{0\}$.
The ascending $J$-compatible series of $\frg$ satisfies
$$\{0\}=\fra_0(J)\subsetneq \fra_1(J)\subsetneq\ldots\subsetneq \fra_{t-1}(J)\subsetneq \fra_{t}(J)=\fra_l(J), \quad{\rm for\ }l\geq t,$$
being~$t$ the smallest integer for which the series stabilizes.
It is worth noting that~$t$ does not necessarily coincide with the nilpotency step~$s$ of the Lie algebra~$\frg$
(see Example \ref{ejemplo8}).
Nonetheless, one always has~$t\leq s$.

We consider the following sequence of quotient Lie algebras
\begin{equation*}
\frg \longrightarrow \frg/\fra_1(J) \longrightarrow \cdots \longrightarrow \frg/\fra_q(J) \stackrel{\pi_{\small q+1}}{\longrightarrow}
\frg/\fra_{q+1}(J) \longrightarrow \cdots \longrightarrow \frg/\fra_t(J),
\end{equation*}
where~$\pi_{q+1}$ is the natural projection onto~$\frg/\fra_{q+1}(J)$, with $\ker\pi_{q+1}=\fra_{q+1}(J)/\fra_q(J)$.
For the seek of simplicity, we will denote~$\widetilde{\frg}_q=\frg/\fra_q(J)$,
for each~$1\leq q \leq t$.
Observe that the Lie algebras~$\widetilde{\frg}_q$ are nilpotent. Moreover, one has that the last term in the
sequence above satisfies~$\widetilde{\frg}_t=\{0\}$
when~$J$ is a complex structure of nilpotent type, and~$\widetilde{\frg}_t\neq\{0\}$ otherwise.

It is easy to see that the complex structure~$J$ defined on~$\mathfrak{g}$ induces in a natural way a
complex structure~$\widetilde{J}_q$ on~$\widetilde{\frg}_q$, for $1\leq q\leq t$, as follows:
$$\widetilde{J}_q(\widetilde{X}) \, \colon\!\!\!= \, \widetilde{JX},\quad\ \ \mbox{ for } \widetilde{X}\in\widetilde{\frg}_q,$$
where $\widetilde{X}$ and $\widetilde{JX}$ denote the classes of $X$ and $JX$, respectively,
in the quotient $\widetilde{\frg}_q$.
Therefore,

\smallskip

 $\bullet$ \ if $J$ is nilpotent, then $(\widetilde{\frg}_t,\widetilde{J}_t)$ is a (complex) abelian Lie algebra;

\smallskip

 $\bullet$ \ if $J$ is weakly non-nilpotent, then $(\widetilde{\frg}_t,\widetilde{J}_t)$ is an NLA endowed with a
 strongly non-nilpotent complex structure.

\medskip

\noindent Hence, in both cases, the given pair $(\frg,J)$ can be recovered as a finite sequence of complex extensions starting
from the lower dimensional pair $(\widetilde{\frg}_t,\widetilde{J}_t)$.
That is to say, quasi-nilpotent complex structures can be constructed from
lower dimensions. Hence, the essentially new complex structures that arise in each even real dimension are those of
strongly non-nilpotent type.

As we have previously said, SnN complex structures appear in dimensions $\geq 6$. In fact, in six dimensions
one has the following structural result:

\begin{theorem}\label{structural-dim6} \cite{U,UV1}
Let $\frg$ be an NLA of dimension $6$. If $\frg$ admits an SnN complex structure, then $\frg$ has ascending type
$(1,3,6)$ or $(1,3,4,6)$.
\end{theorem}

This allows us to immediately assert that an NLA $\frg$ with ascending type, for instance, $(1,4,6)$ does not
admit any complex structure. Indeed, if we assume the converse, then the complex structure~$J$ on~$\frg$ must be
SnN, as the center of~$\frg$ is $1$-dimensional and consequently~$\fra_1(J)=\{0\}$; however,
this is not possible by Theorem~\ref{structural-dim6}.

Our goal in the next sections is to find, in any even dimension, new restrictions on the ascending type $(m_{1},\ldots,m_{s})$ of
an NLA $\frg$ endowed with a strongly non-nilpotent complex structure $J$.
One of the main consequences of our study is a structural result for 8-dimensional NLAs
admitting an SnN complex structure (see Theorem~\ref{teorema-estructura-acs}), thus providing an analogous
result to Theorem~\ref{structural-dim6} in eight dimensions.

We finish this section by noting that the existence of a complex structure on an NLA is characterized in \cite{S} in terms
of the existence of a certain complex (1,0)-basis for $\frg_{\mathbb C}^*$.
However, we will not make use of this point of view here, and we
will instead develop a more constructive approach to the problem.


\section{Strongly non-nilpotent complex structures}\label{SnN}

\noindent
Let us focus our attention on those complex structures $J$ on $\frg$ satisfying $\fra_1(J)=\{0\}$.
In this case, the construction of the pairs $(\frg,J)$ cannot be based on lower dimensional classifications and other
approach is needed. In this section, we will show that these $(\frg,J)$
can be found
using the ascending central series and constructing an appropriate $J$-adapted basis of $\frg$.
In the first part of the section, we provide some general results about the ascending central series $\{\frg_k\}_k$ of those NLAs $\frg$ admitting SnN complex structures.
In particular, we see that the nilpotency step of $\frg$ must be at least~$3$.
Then, we focus on the term $\frg_1$, which corresponds to the center of $\frg$, and find an upper bound for its dimension. In the last part,
our results are applied to the study of existence of complex structures
on products of NLAs.

\subsection{Restrictions on the ascending central series $\{\frg_k\}_k$}\label{sec:restrictions}

We introduce here a collection of results showing that the ascending
central series of an NLA $\frg$ is restricted under the existence of a complex structure $J$ on $\frg$.
We also obtain several
algebraic obstructions for an NLA~$\frg$ to admit complex structures.

Let us start with some results where $J$ can be a complex structure of any type (recall Definition \ref{tipos_J}).


\begin{lemma}\label{lemma1_serie_asc}
Let $(\frg,J)$ be a $2n$-dimensional NLA  endowed with a complex structure.
Suppose that there exists a subspace $\mathcal W\subset\frg_k$, where $k>1$,
such that dim\,$\mathcal W=n$ and $\mathcal W\cap J\mathcal W=\{0\}$. It holds:
\begin{enumerate}
\item[(i)] if $\frg_{k-1}=J\frg_{k-1}$, then $\frg_k=\frg$;

\item[(ii)] if there is $X\in \mathcal W$ such that $JX\in\frg_{k+1}$ and $JX\notin\frg_{k}$, then it exists $Y\in\frg_{k-1}$ such that
$JY\in\frg_k$ and $JY\notin\frg_{k-1}$.
\end{enumerate}
\end{lemma}

\begin{proof}
Let $\{X_i\}_{i=1}^n$ be a basis for $\mathcal W$.  Observe that $X_i\notin J\mathcal W$, for each
$1\leq i \leq n$, and $\{X_i, JX_i\}_{i=1}^n$ gives a basis of $\frg$.
First, let us note that $X_i\in\frg_k$, for every $1\leq i \leq n$, implies that
$$[X_i,X_j], \ [X_i, JX_j] \in \frg_{k-1}, \ \forall\, j=1,\ldots,n.$$
For the brackets $[JX_i,JX_j]$, considering the  Nijenhuis condition~\eqref{Nijenhuis} and the hypothesis in~\textrm{(i)}, we obtain
$$[JX_i,JX_j]=[X_i,X_j]+J[JX_i,X_j]+J[X_i,JX_j] \ \in \ \frg_{k-1}+J\frg_{k-1}=\frg_{k-1},$$
for all $1\leq i,j \leq n$. Thus, we can easily conclude that $JX_1,\ldots,JX_n\in\frg_k$ and $\frg_k=\frg$.

\smallskip
For the second part, let us suppose that there exists $X\in \mathcal W\subset\frg_k$ such that~$JX\in\frg_{k+1}$ and~$JX\notin\frg_{k}$.
Then, it is possible to find~$V\in \frg$ such that~$0\neq [JX,V]\in\frg_k$ but~$[JX, V]\notin\frg_{k-1}$. As vector spaces, $\frg=\mathcal W\oplus J\mathcal W$ so we can express $V=V_1+V_2$, where
$V_1\in\mathcal W$ and $V_2\in J\mathcal W$. Hence,
$[JX, V] = [JX, V_1] + [JX, V_2]$, with $[JX,V_1]\in\frg_{k-1}$, and
consequently $[JX,V_2]$ must verify the same property as~$V$.
Writing~$V_2=JZ$, with~$Z\in\mathcal W$, one has that
$0\neq [JX,JZ]\in\frg_k$ but $[JX, JZ]\notin\frg_{k-1}$.
Applying the Nijenhuis condition, we get
$$[JX,JZ]-[X,Z]=J\big([JX,Z]+[X,JZ]\big).$$
Observe that the left-hand side belongs to $\frg_k$, but it is not contained in $\frg_{k-1}$.
Hence, the element $0\neq Y=[JX,Z]+[X,JZ]\in\frg_{k-1}$ and $JY\in\frg_k$, but it does not belong to $\frg_{k-1}$.
\end{proof}


\begin{lemma}\label{lemma2_serie_asc}
Let $(\frg,J)$ be a $2n$-dimensional NLA  endowed with a complex structure. Suppose there exist a $2$-dimensional $J$-invariant subspace $\mathcal V\subset\frg$ and a non-zero element $X\in\frg_{1}$
such that $\frg=\frg_{k}\oplus \mathcal V \oplus \langle JX\rangle$ as vector spaces, for some $k\geq 1$. Then, $JX\in\frg_{k+1}$.
\end{lemma}

\begin{proof}
Let $\{X_i,JX_i\}_{i=1}^n$ be a $J$-adapted basis of $\frg$.
Then, any non-zero element $W\in\frg$ can be written as $W=\sum_{i=1}^n a_i\,X_i+b_i\,JX_i$,
where $a_k\neq 0$ for some $1\leq k\leq n$ (indeed, if $a_k=0$ for every $k$, then $b_k\neq 0$ for some $k$ and it suffices
to interchange the roles of $X_k$ and $JX_k$ in the $J$-adapted basis).

We consider $X\in\frg_1$ and $\mathcal V=\langle Y,JY\rangle$. We have
$$X=\sum_{i=1}^n a_i\,X_i+b_i\,JX_i,
    \quad
    Y=\sum_{i=1}^n \tilde a_i\,X_i+\tilde b_i\,JX_i, \ \text{ with } a_r,\,\tilde a_s\neq 0 \text{ for some } r\neq s,$$
since $X, Y$ are non-zero linearly independent elements.
The following arrangement
can be made
$$X'_1=X, \qquad X'_r=X_1, \qquad X'_s=X_n, \qquad X'_n=Y, \qquad X'_i=X_i, \text{ for } i\neq 1,k,s,n,$$
in such a way that $\{X'_i,JX'_i\}$ is a new $J$-adapted basis of $\frg$.
Renaming $X'_k\equiv X_k$ for $1\leq k\leq n$, we have
that $X_{1}=X$ and $\mathcal V=\langle X_{n},JX_{n}\rangle$. Hence, it is clear that
$$\frg_{k}=\langle X_{1},\ldots,X_{n-1},JX_{2},\ldots,JX_{n-1}\rangle.$$

Since the algebra is nilpotent and $\frg_k\neq \frg$, there exists $0\neq Z\in\frg_{k+1}$ such that $Z\notin\frg_k$.  We can consider, without loss of generality, that $$Z=a\,X_n + b\,JX_n + c\,JX_1,$$ where $a, b, c\in \mathbb R$
(otherwise, apply a similar argument to that contained in the proof of Lemma~\ref{lemma1_serie_asc}).

If $a=b=0$, then $c\neq 0$ and the result holds immediately.

If $a$ or $b$ are non-zero, then we can define a new basis for $\mathcal V$ as follows:
$$X_n' = b\,X_n - a\,JX_n,\quad JX_n' = b\,JX_n + a\,X_n.$$
In order to prove that $JX_1\in\frg_{k+1}$ it suffices to see that $[X_n', JX_1], [JX_n', JX_1]\in \frg_k$.

On the one hand, it is clear that $[Z, JX_1]\in \frg_k$, i.e.
$$[Z, JX_1] = a_{1}\,X_{1}+\cdots+a_{n-1}\,X_{n-1}+b_{2}\,JX_{2}+\cdots +b_{n-1}\,JX_{n-1},$$ where $a_{i},b_{i}\in\mathbb R$.
On the other hand, note that we can write $Z = J(X_n' + c\,X_1)\in \frg_{k+1}$, so using the Nijenhuis condition we obtain
$$[Z, JX_1] = J[c\,X_1 + X_n', JX_1] = J[X_n', JX_1].$$
Therefore,
$$J[X_n', JX_1] = a_{1}\,X_{1}+\cdots+a_{n-1}\,X_{n-1}+b_{2}\,JX_{2}+\cdots +b_{n-1}\,JX_{n-1},$$
and applying $J$ to the previous expression, we have:
$$-[X_n', JX_1] = a_{1}\,JX_{1}+ a_2\,JX_2+\cdots+a_{n-1}\,JX_{n-1}-(b_{2}\,X_{2}+\cdots +b_{n-1}\,X_{n-1}).$$  Due to the nilpotency of $\frg$, we get $a_1 = 0$ and thus $[X_n', JX_1]\in \frg_k$.

Finally, since $[JX_n', JX_1] = J [X_n', JX_1]$, we can also ensure that $[JX_n', JX_1]\in\frg_k$.
\end{proof}

\begin{remark}
In the conditions of Lemma~\ref{lemma2_serie_asc}, we notice that if $\frg_{k+1}\cap \mathcal V\neq \{0\}$, then $\frg_{k+1}=\frg$ whereas if $\frg_{k+1}\cap \mathcal V= \{0\}$, then $\frg_{k+1} = \frg_{k}\oplus \langle JX \rangle$ and $\frg_{k+2}=\frg$.
\end{remark}

In the case of SnN complex structures, the ascending $J$-compatible series~$\{a_k(J)\}_k$ is identically zero.
For this reason, in order to study this type of structures we will focus on the way~$J$ interacts
with the terms~$\frg_k$ of the ascending central series of $\frg$.


\begin{lemma}\label{lemma1_SnN}
Let $(\frg,J)$ be a $2n$-dimensional NLA  endowed with a complex structure,
and assume that
$\frg_k\cap J\frg_k=\{0\}$, for some integer $k\geq 1$. Then one has:
\begin{itemize}
\item[\textrm{(i)}] $\frg_{k+1}\cap J\frg_k=\{0\}$;
\item[\textrm{(ii)}] if $r>1$ is the smallest integer such that $\frg_{k+r}\cap J\frg_k\neq \{0\}$, then
$\frg_{k+r-1}\cap J\frg_{k+r-1}\neq \{0\}.$
\end{itemize}
\end{lemma}
\begin{proof}
For case (i) we argue by contradiction.  Suppose that there exists $X\in\frg_k$ such that $JX\in\frg_{k+1}$ and $JX\notin\frg_{k}$. Then, it is possible to find $Y\in\frg$ satisfying $0\neq [JX,Y]\in\frg_{k}$ but not contained
in the subspace $\frg_{k-1}$.
Let us denote by $T$ the non-zero element given by
$T=[JX,Y]+[X,JY]\in\frg_{k}$.
By the Nijenhuis condition, we have
$$JT=J\big([JX,Y]+[X,JY]\big)=[JX,JY]-[X,Y]\in\frg_k.$$
Therefore, both the element $T$ and its image by $J$ belong to $\frg_{k}$, which contradicts the hypothesis~$\frg_k\cap~\!J\frg_k=\{0\}$.

For case (ii), take an element $X\in\frg_k$ such that $JX\in\frg_{k+r}$ and $JX\notin\frg_{k+r-1}$.  Following the same argument as before, we can find an element $Y\in\frg$ satisfying $0\neq [JX,Y]\in\frg_{k+r-1}$ but not contained in the subspace $\frg_{k+r-2}$.  Now, the non-zero elements
$T$ and $JT$ belong to $\frg_{k+r-1}$, being $T=[JX,Y]+[X,JY]$.
\end{proof}

%

Some interesting consequences arise from this result:

\begin{proposition}\label{corolario_step}
Let $(\frg,J)$ be a $2n$-dimensional NLA  endowed with a complex structure and suppose that $\frg_k\cap J\frg_k=\{0\}$,
for some integer $k\geq 1$.  Then:
\begin{itemize}
\item[\textrm{(i)}] the nilpotency step of $\frg$ satisfies $s\geq k+2$; moreover,
if $s=k+2$ then $\frg_{k+1}\cap J\frg_{k+1}\neq\{0\}$;
\item[\textrm{(ii)}] $\frg_{k+1}$ cannot contain a subspace $\mathcal W$ satisfying dim\,$\mathcal W=n$ and $\mathcal W\cap J\mathcal W=\{0\}$.
\end{itemize}
\end{proposition}

\begin{proof}
For part \textrm{(i)}, it suffices to apply Lemma \ref{lemma1_SnN} \textrm{(i)} in order to see that $\frg_{k+1}\neq\frg$.
Hence, $s\geq k+2$.
Furthermore, if $s=k+2$,
then one has $\frg_{k+2}=\frg$ and also $\frg_{k+1}\cap J\frg_{k}=\{0\}$ by Lemma \ref{lemma1_SnN} \textrm{(i)}.
To obtain our result, it suffices to make use of Lemma~\ref{lemma1_SnN}~\textrm{(ii)} with~$r=2$.

For  \textrm{(ii)}, let us assume the opposite, i.e., suppose that there exists $\mathcal W\subset\frg_{k+1}$
such that dim\,$\mathcal W=n$ and $\mathcal W\cap J\mathcal W=\{0\}$.
By part \textrm{(i)} it is clear that $\frg_{k+1}\neq\frg$.
Necessarily, there is $Z\in\frg_{k+2}\subseteq\frg$
 and not belonging to $\frg_{k+1}$.
As vector spaces, $\frg=\frg_{k+1}\oplus \mathcal U$, where $\mathcal U$ is a certain
subspace of $J\mathcal W$, so we can write $Z=Z_1+JX$, being $Z_1\in\frg_{k+1}$ and $0\neq X\in\mathcal W$.
Moreover, since $Z\notin\frg_{k+1}$, then $JX\in\frg_{k+2}$ but $JX\notin\frg_{k+1}$.
Applying Lemma~\ref{lemma1_serie_asc}~\textrm{(ii)} for $r=k+1$,
it would be possible to find a non-zero element $Y\in\frg_k$ such that $JY\in\frg_{k+1}$.
However, this contradicts Lemma \ref{lemma1_SnN}~\textrm{(i)}.
\end{proof}

Thanks to the previous proposition with $k=1$, one obtains the following result:

\begin{corollary}\label{corolario_smayor3}
Any nilpotent Lie algebra endowed with a strongly non-nilpotent complex structure is at least 3-step.
\end{corollary}


We next give some bounds for the dimensions of the terms of the ascending central series of $\frg$ under the presence of a strongly non-nilpotent complex structure.

\begin{proposition}\label{dim-g}
Let $(\frg,J)$ be a $2n$-dimensional NLA  endowed with a complex structure satisfying $\frg_k\cap J\frg_k=~\{0\}$ for some $k\geq 1$.  Then:
\begin{itemize}
\item[(i)] $k\leq \dim \frg_k\leq n-2$;
\item[(ii)] $1+\dim \frg_{k}\leq \dim \frg_{k+1}\leq 2n-3$;
furthermore, if in addition $\frg_{k+1}\cap J\frg_{k+1}\neq\{0\}$, then one indeed has $2+\dim \frg_{k}\leq \dim \frg_{k+1}\leq 2n-3$;
\item[(iii)] if $\dim \frg_k=n-r$, for some $r\geq 1$, then $\frg_{k+r-1}\cap J\frg_{k+r-1}\neq\{0\}$;
moreover, $n\leq \dim \frg_{k+r-1}$.
\end{itemize}
\end{proposition}

\begin{proof}
For part (i) observe that the lower bound is clear, because $\frg$ is a nilpotent Lie algebra and its ascending central series strictly
increases.
For the upper bound, first note that the hypothesis $\frg_k\cap J\frg_k=\{0\}$ leads to dim\,$\frg_k\leq n$.
Thus, we just need to discard the cases dim\,$\frg_k=n$ and dim\,$\frg_k=n-1$.

Let us first suppose that dim\,$\frg_k=n$. By Lemma \ref{lemma1_SnN} \textrm{(i)}, one has $\frg_{k+1}\cap J\frg_k=\{0\}$, so the dimension of the space $\frh=\frg_{k+1}\oplus J\frg_k\subseteq\frg$ must satisfy
$(n+1)+n\leq\text{dim\,}\frg_{k+1} + \text{dim\,}J\frg_k = \text{dim\,}\frh \leq \text{dim\,}\frg=2n$. As we can see, this is a contradiction.

Let us now consider dim\,$\frg_k=n-1$.
The nilpotency of $\frg$ guarantees the existence of some $Y\in\frg_{k+1}$ not contained in $\frg_k$.
Furthermore, as a consequence of Lemma \ref{lemma1_SnN} part \textrm{(i)} we have that $Y\neq JX$, for any $X\in\frg_k$.
Hence, one has $\frg_k\cap \langle Y \rangle=\{0\}$, and $\mathcal W=\frg_k\oplus\langle Y\rangle\subseteq \frg_{k+1}$ is an $n$-dimensional subspace satisfying $\mathcal W\cap J\mathcal W=\{0\}$.
However, this contradicts Proposition~\ref{corolario_step} \textrm{(ii)}.
Therefore, we conclude $k\leq$ dim\,$\frg_{k}\leq n-2$.

In part (ii), the lower bound comes straightforward.
For the upper bound, we first observe that Proposition~\ref{corolario_step}~\textrm{(i)}
and the nilpotency of $\frg$
imply that~$\dim \frg_{k+1}\leq 2(n-1)$. Consequently, it suffices to discard the case~$\dim \frg_{k+1}=2(n-1)$.
Let us then assume that~$\dim \frg_{k+1}=2(n-1)$.
Due to Proposition~\ref{corolario_step}~\textrm{(ii)}, notice that~$\frg_{k+1}$ is $J$-invariant
and thus $J\frg_{k+1}=\frg_{k+1}$. Therefore,
$$\{0\}=\frg_{k+1}\cap J\frg_k = J\frg_{k+1}\cap \frg_k = \frg_{k+1}\cap \frg_k = \frg_k,$$
which
is a contradiction.

If we now assume the additional hypothesis $\frg_{k+1}\cap J\frg_{k+1}\neq\{0\}$, then
there exists $0\neq X\in\frg$ such that $X\in\frg_{k+1}\cap J\frg_{k+1}$.  Moreover, since $\frg_{k+1}\cap J\frg_{k+1}$ is a $J$-invariant space, also $0\neq JX\in \frg_{k+1}\cap J\frg_{k+1}$.  In particular,
one has a $2$-dimensional subspace
$\mathcal V=\langle X,\, JX\rangle \subset \frg_{k+1}\cap J\frg_{k+1}$. In addition, since $\frg_{k+1}\cap J\frg_{k}=\{0\}$
by Lemma \ref{lemma1_SnN} (i), we can ensure
that $\mathcal V\cap \frg_k=\{0\}$ and the result holds.

Finally, let us prove part (iii).  Observe that if $\frg_{k+r-2}\cap J\frg_{k+r-2}\neq\{0\}$, then the result is trivial. Hence, let us focus on the case
$\frg_{k+r-2}\cap J\frg_{k+r-2}=\{0\}$. On the one hand,
by part (i), one has $\dim \frg_{k+r-2}\leq n-2$.
On the other hand, we know that $n-r=\dim \frg_k<\dim \frg_{k+1} < \ldots < \dim \frg_{k+r-2}$.
Since these inequalities are strict, it is easy to see that
$\dim \frg_{k+r-2}\geq n-2$. Joining these two facts, we can conclude that $\dim \frg_{k+r-2}=n-2$
and thus, $\dim \frg_{k+r-1}\geq n-1$.
Using (i) again, the result comes straightforward.
For the second part of the statement, use part (ii), bearing in mind that the given hypothesis
imply that there should be an integer $k\leq t \leq k+r-2$ such that $\frg_{t}\cap J\frg_{t}=\{0\}$
and $\frg_{t+1}\cap J\frg_{t+1}\neq \{0\}$.
\end{proof}

As an immediate consequence of part (ii) we have:

\begin{corollary}\label{k-k+1-sin-Js}
Let $(\frg,J)$ be a $2n$-dimensional NLA  endowed with a complex structure, where $n\geq3$.
If $\frg_k\cap J\frg_k=\{0\}$ and $\dim \frg_{k+1}=1+\dim \frg_{k}$, for some $k\geq 1$,
then $\frg_{k+1}\cap J\frg_{k+1}=\{0\}$.
\end{corollary}

Using the result above, we reach an algebraic obstruction to the existence of complex structures:


\begin{corollary}\label{dim-no-compleja}
Let $\frg$ be a $2n$-dimensional NLA.
If $\dim \frg_{n-1}=n-1$, then $\frg$ does not admit any complex
structure.
\end{corollary}

\begin{proof}
First note that the nilpotency of $\frg$ implies that $\frg_{k}\subsetneq \frg_{k+1}$, for each $1\leq k \leq s-1$, being $s$ the nilpotency step of $\frg$.
In fact, using the hypothesis we have that $\dim \frg_{k}=k$, for every $1\leq k\leq n-1$, or equivalently, $\dim\,\frg_{k+1}=1+\dim\,\frg_k$.

Suppose that there exists a complex structure $J$ on $\frg$.
Since dim\,$\frg_1=1$, clearly
$\frg_{1}\cap J\frg_{1}=\{0\}$, so $J$ must be of SnN type. For $n=2$, it is well known that there are no SnN complex structures on NLAs.  For $n\geq 3$,
we can apply Corollary~\ref{k-k+1-sin-Js} with~$k=1$ to get $\frg_{2}\cap J\frg_{2}=\{0\}$, but we are again in the conditions of Corollary~\ref{k-k+1-sin-Js}, this time with $k=2$.
Repeating the process,
one finally gets $\frg_{n-1}\cap J\frg_{n-1}=\{0\}$.
However, this contradicts Proposition~\ref{dim-g}, part \textrm{(i)}.
\end{proof}

We notice that this result provides a restriction to the existence of complex structures based on the ascending type of the NLA:
if a $2n$-dimensional Lie algebra $\frg$ has ascending type $(1,2,\ldots,n-1,m_n,\ldots,m_s)$, then
$\frg$ does not admit any complex structure.

As a consequence, we recover the following result proven by Goze and Remm in \cite{GR}.
Recall that an $m$-dimensional filiform Lie algebra is an $(m-1)$-step
nilpotent Lie algebra.


\begin{corollary}\label{dim-no-compleja-fil}
Filiform Lie algebras do not admit complex structures.
\end{corollary}

\begin{proof}
Any $2n$-dimensional filiform Lie algebra $\frg$ has ascending type $(1,2,\ldots,n-1,\ldots,2n-2,2n)$.
Hence, $\dim \frg_{n-1}=n-1$ and the result is a direct consequence of Corollary~\ref{dim-no-compleja}.
\end{proof}

\subsection{Dimension of the center}\label{proof}

The goal of this section is to obtain an upper bound for the dimension of the center of a $2n$-dimensional NLA $\frg$ endowed with an SnN complex structure $J$.
More precisely, we improve the bound given in Proposition~\ref{dim-g} (i) for $k=1$ when $n\geq 4$.

\begin{theorem}\label{prop_centro}
Let ($\frg, J$) be a $2n$-dimensional nilpotent Lie algebra with $n\geq 4$ endowed with an strongly non-nilpotent complex structure $J$.
Then, $1\leq \dim\frg_1 \leq n-3$.
\end{theorem}

\begin{proof}
As a consequence of Proposition~\ref{dim-g} \textrm{(i)}, it suffices to discard the case dim\,$\frg_1 = n-2$.
We will prove the theorem by contradiction, using the following idea.
As a starting point, we consider a $2n$-dimensional vector space $\frg$ with an almost-complex structure~$J$.
We will try to endow $\frg$ with the structure of a nilpotent Lie algebra with $\text{dim}\,\frg_1=n-2$, assuming
that $J$ must be integrable and of SnN type. In particular, we will first impose that ${\mathfrak g}$ has an $(n-2)$-dimensional center
and then find all possible combinations for the remaining terms in the ascending central
series $\{ {\mathfrak g}_k \}_k$. This will be done defining the Lie brackets of ${\mathfrak g}$
in terms of the elements of a basis $\mathcal B$ that will be constructed along the process,
attending to the nilpotency of ${\mathfrak g}$ and the Jacobi identity
\begin{equation} \label{Jacobi}
Jac\,(X,Y,Z) \ := \ \big[[X,Y],Z \big] + \big[[Y,Z],X\big] + \big[[Z,X],Y \big] \ = \ 0,
\end{equation}
for any $X,Y,Z\in\frg$,
the Nijenhuis
condition~\eqref{Nijenhuis}, and the strongly non-nilpotency of $J$.
These four conditions are checked at every stage of the method, discarding the cases in which
any of them fails. In the end, we will see that all the possible cases can be rejected.

Let us start with the proof.
Suppose that $\dim \frg_1=n-2$ and consider
$$\frg_1=\langle X_1,\ldots,X_{n-2}\rangle.$$
In this way, we can choose $X_1,\ldots,X_{n-2}$ as generators in $\mathcal{B}$.
By the Nijenhuis condition, we have
\footnote{ {\sc Notation:} By $[X_k,\cdot\ ]_{1\leq k\leq n-2}$ we denote all the brackets $[X_k,\cdot\ ]$, where $1\leq k\leq n-2$.}
$$[X_k,\cdot\ ]_{1\leq k\leq n-2}=0, \qquad [JX_k,JX_r]_{1\leq k < r \leq n-2}=0,$$
since $X_1,\ldots,X_{n-2}$ are in the center of $\frg$.

We want to complete $\{X_{k},JX_{k}\}_{k=1}^{n-2}$ up to a $J$-adapted basis $\mathcal B$ of $\frg$, for which it suffices to find two vectors linearly independent with the previous ones that will be called $X_{n-1}$ and $X_{n}$.
Observe that the yet completely undetermined brackets for $\frg$ are exactly those involving
at least one of the elements $X_{n-1},$ $X_n,$ $JX_{n-1},$ $JX_n\in\mathcal{B}$.

We should focus on $\frg_2$. Our first observation is that $\frg_2\cap J\frg_1=\{0\}$, due to
Lemma~\ref{lemma1_SnN}~\textrm{(i)}. Moreover, we know that $\dim\frg_2\geq n-1$,
so applying Proposition~\ref{dim-g}, part \textrm{(i)} with $k=2$, one immediately has $\frg_2\cap J\frg_2\neq\{0\}$,
and thus $\dim\frg_2 \geq n$ (see Proposition~\ref{dim-g}~\textrm{(ii)}).
Hence, by Proposition~\ref{corolario_step}, part \textrm{(ii)} for $k=1$, we conclude that $\dim\frg_2=n$.
As a consequence $\frg_2 = \frg_1\oplus \mathcal V$, being $\mathcal V$ a $J$-invariant 2-dimensional space.
We can take a basis $\{X, JX\}$ of $\mathcal V$ and then choose $X_{n-1} = X$. Then,
$$\frg_1=\langle X_1,\ldots,X_{n-2}\rangle, \quad \frg_2=\langle X_1,\ldots,X_{n-1},JX_{n-1}\rangle.$$
Since the elements $X_{n-1}, JX_{n-1}$ of the basis $\mathcal B$ determined above belong to $\frg_2$,
we are able to set the following real brackets:
$$[X_{n-1},JX_k]_{1\leq k\leq n-2} \ = \ \sum\nolimits_{i=1}^{n-2} \mu_i^k\,X_i,\qquad\quad
[X_{n-1},JX_{n-1}] \ = \ \sum\nolimits_{i=1}^{n-2} \beta_i\,X_i.$$
Applying the Nijenhuis condition,
$$
[JX_{k},JX_{n-1}]_{1\leq k\leq n-2} \ = \ -\sum\nolimits_{i=1}^{n-2} \mu_i^k\,JX_i.$$
These brackets belong to $\frg_{1}$, so necessarily $\mu_i^k=0$ for every $1\leq i,k \leq n-2$.

Up to this moment, we have:
\begin{equation}\label{brackets-centro}
\begin{array}{ll}
[X_k,\cdot\ ]_{1\leq k\leq n-2}=0,
  & [X_{n-1},JX_{n-1}] \ = \ \sum_{i=1}^{n-2} \beta_i\,X_i,\\[4pt]
[X_{n-1},JX_k]_{1\leq k\leq n-2} \ = 0,  \quad
  & [JX_k,JX_r]_{1\leq k < r \leq n-1}=0.
\end{array}
\end{equation}
Since $\frg_1$ and $\frg_2$ are already determined, let us move to $\frg_3$.  We know that dim\,$\frg_3\geq n+1$, due to the nilpotency of $\frg$. First suppose that $\frg_3\cap J\frg_1=\{0\}$.  \label{contra}  Then,
there should be a vector in $\frg_{3}$ linearly
independent with $\{X_{i},\, JX_{i}\}_{i=1}^{n-1}$.
We select this vector as an element of $\mathcal B$ and denote it by $X_{n}$.
Moreover, due to the nilpotency of $\frg$
one could find $Z\in \mathcal W=\langle X_{1},\ldots,X_{n}\rangle\subset\frg_{3}$
such that $JZ\in\frg_{4}$ and $JZ\notin\frg_{3}$. Applying
Lemma \ref{lemma1_serie_asc} \textrm{(ii)} with $r=3$, there exists an element $Y\in\frg_{2}$ such that
$JY\in\frg_{3}$ and $JY\notin\frg_{2}$, but this is not possible by construction.

Therefore, we may assume that $\frg_3\cap J\frg_1\neq \{0\}$, i.e. there is a vector $X\in\frg_1$ such that  $JX\in\frg_3$.  We next see that it is possible to set $JX_1\in \frg_3$, without loss of generality. Notice that a generic non-zero $X\in \frg_1$ can be expressed as $X=\sum_{i=1}^{n-2}s_i\,X_i$, with $(s_1,\ldots,s_{n-2})\neq (0,\ldots,0)$.
Since $JX$ is a generator of $\frg_{3}$, we would like to include it in $\mathcal B$.
However, as we want $\mathcal B$ to be $J$-adapted, also $X$ should belong to $\mathcal B$, so
we must arrange its generators in order to fulfill this requirement. We can do
this as follows:
\begin{itemize}
\item if $s_1\neq 0$, consider $X'_1=X$ and $X'_k=X_k$, for every $2\leq k \leq n-1$; \label{arrangement}
\item if $s_1=\cdots=s_r=0$, for some $1\leq r\leq n-3$, and $s_{r+1}\neq 0$, then choose $X'_1=X$, $X'_{r+1}=X_1$,
and $X'_k=X_k$, for
every $2\leq k \leq n-1$ such that $k\neq r+1$.
\end{itemize}
Let us notice that this type of changes do not affect the structure of the ascending
central series that has been adjusted up to this moment.
In fact, the brackets of the elements~$\{X'_k,JX'_k\}_{k=1}^{n-1}$
still follow~\eqref{brackets-centro}, maybe modifying the coefficients~$\beta_i$ if necessary (which were anyway free).
Furthermore, we get $JX'_1\in\frg_3$.

This fact allows us to conclude that our assumption is equivalent to $JX_1\in\frg_3$, \emph{up to arrangement of generators}.  In this way, we have
\begin{equation}\label{serie1}
\frg_1=\langle X_1,\ldots,X_{n-2}\rangle, \ \,
  \frg_2=\langle X_1,\ldots,X_{n-1},JX_{n-1}\rangle, \ \,
  \frg_3\supseteq\langle X_1,\ldots,X_{n-1},JX_1,JX_{n-1}\rangle.
  \end{equation}
Since $JX_1\in\frg_3$ but $JX_{1}\notin\frg_2$, there must exist an element
$Y\in\frg$ such that $[JX_1,Y]\in\frg_2$ and satisfying $[JX_1,Y]\notin\frg_1$.
In view of the brackets \eqref{brackets-centro}, this element $Y$ should be linearly independent with
$\{X_{i},\, JX_{i}\}_{i=1}^{n-1}$. Hence, we can choose $X_{n}=Y$ as a
new generator in $\mathcal{B}$. Observe that $X_{n}\in\frg_{k}$ and $X_n\notin\frg_{2}$ for some $k\geq 3$. Using the fact that $JX_1\in\frg_3$, we can take:
$$[X_n,JX_1]=\sum\nolimits_{i=1}^{n-2} b_i^1\,X_i+b_{n-1}^1\,X_{n-1}+B_{n-1}^1\,JX_{n-1},$$
where $b_{i}^{1},B_{n-1}^{1}\in\mathbb{R}$, for $1\leq i \leq n-1$. Let us remark that one needs
$(b_{n-1}^1,B_{n-1}^1)\neq (0,0)$ in order to ensure
$X_{n}\notin\frg_{2}$. Furthermore, the Nijenhuis condition yields
$$[JX_1,JX_n]=-\sum\nolimits_{i=1}^{n-2} b_i^1\,JX_i-b_{n-1}^1\,JX_{n-1}+B_{n-1}^1\,X_{n-1} \ \in \ \frg_2,$$
and thus $b_i^1=0$, for every $1\leq i \leq n-2$.  Moreover, $X_{n-1},\, JX_{n-1}\in\frg_2$, so
$$
[X_{n-1},X_n] = \sum\nolimits_{i=1}^{n-2} \alpha_i\,X_i, \quad
[X_{n-1},JX_{n}] = \sum\nolimits_{i=1}^{n-2} \gamma_i\,X_i,\quad
[X_n,JX_{n-1}] = \sum\nolimits_{i=1}^{n-2}a_i\,X_i,
$$
where $\alpha_{i},\,a_{i},\,  \gamma_{i}\in\mathbb{R}$, for $1\leq i,k\leq n-2$.
If we now apply the Nijenhuis
condition, then we get
$$[JX_{n-1},JX_{n}] \ = \ \sum\nolimits_{i=1}^{n-2} \alpha_i\,X_i + \sum\nolimits_{i=1}^{n-2} (\gamma_i-a_{i})\,JX_i\in \frg_1.$$
Necessarily $a_i=\gamma_i$, for all $1\leq i,k \leq n-2$.

Our collection of Lie brackets is now:
\begin{equation}\label{brackets-centro-1}
\begin{array}{lcl}
[X_k,\cdot\ ]_{1\leq k\leq n-2} \ = \ 0,
    & & [X_n,JX_{n-1}]=\sum\nolimits_{i=1}^{n-2}\gamma_i\,X_i, \\[2pt]
[X_{n-1},X_n] \ = \ \sum\nolimits_{i=1}^{n-2} \alpha_i\,X_i,
    & & [X_n,JX_n] \text{ unknown}, \\[2pt]
[X_{n-1},JX_k]_{1\leq k\leq n-2} \ = \ 0,
    & & [JX_k,JX_r]_{1\leq k < r\leq n-1} \ = \ 0,\\[2pt]
[X_{n-1},JX_{n-1}] \ = \ \sum\nolimits_{i=1}^{n-2} \beta_i\,X_i,
    & & [JX_1,JX_n] \ = \ B^1_{n-1}\,X_{n-1} - b^1_{n-1}\,JX_{n-1},\\[2pt]
[X_{n-1},JX_n] \ = \ \sum\nolimits_{i=1}^{n-2} \gamma_i\,X_i,
    & & [JX_k,JX_n]_{2\leq k\leq n-2} \text{ unknown}, \\[2pt]
[X_n,JX_1] \ = \ b^1_{n-1}\,X_{n-1} + B^1_{n-1}\,JX_{n-1},
    & & [JX_{n-1},JX_n] \ = \ \sum\nolimits_{i=1}^{n-2} \alpha_i\,X_i,\\[2pt]
[X_n,JX_k]_{2\leq k\leq n-2} \text{ unknown},
    & &
\end{array}
\end{equation}
where $\alpha_{i}, \beta_{i}, \gamma_{i}\in\mathbb{R}$, for $1\leq i \leq n-2$.  Let us remark that only the values preserving the dimension of the ascending
central series~\eqref{serie1} are valid for our purposes.

Recall that we also have to ensure that the Jacobi identity \eqref{Jacobi} holds for any triplet $(X, Y, Z)$.
Computing $Jac(X_n, JX_1, Z)$, for $Z= X_{n-1}, J X_{n-1}$,
we obtain:
$$b^1_{n-1}\beta_i = B^1_{n-1}\beta_i = 0,\quad 1\leq i \leq n-2.$$
Since $(b^1_{n-1}, B^1_{n-1}) \neq(0,0)$, then $\beta_i=0$ for $1\leq i \leq n-2.$  Using these conditions, \eqref{brackets-centro-1} become:
\begin{equation}\label{brackets-centro-2}
\begin{array}{lcl}
[X_k,\cdot\ ]_{1\leq k\leq n-2} \ = \ 0,
    & & [X_n,JX_{n-1}]=\sum\nolimits_{i=1}^{n-2}\gamma_i\,X_i, \\[2pt]
[X_{n-1},X_n] \ = \ \sum\nolimits_{i=1}^{n-2} \alpha_i\,X_i,
    & & [X_n,JX_n] \text{ unknown}, \\[2pt]
[X_{n-1},JX_k]_{1\leq k\leq n-1} \ = \ 0,
    & & [JX_k,JX_r]_{1\leq k < r\leq n-1} \ = \ 0,\\[2pt]
[X_{n-1},JX_n] \ = \ \sum\nolimits_{i=1}^{n-2} \gamma_i\,X_i,
    & & [JX_1,JX_n] \ = \ B^1_{n-1}\,X_{n-1} - b^1_{n-1}\,JX_{n-1},\\[2pt]
[X_n,JX_1] \ = \ b^1_{n-1}\,X_{n-1} + B^1_{n-1}\,JX_{n-1},
    & & [JX_k,JX_n]_{2\leq k\leq n-2} \text{ unknown}, \\[2pt]
[X_n,JX_k]_{2\leq k\leq n-2} \text{ unknown},
    & & [JX_{n-1},JX_n] \ = \ \sum\nolimits_{i=1}^{n-2} \alpha_i\,X_i,
\end{array}
\end{equation}
where the coefficients are real numbers preserving the ascending type and satisfying the Jacobi identity.

Let us now consider  $Jac\,(X_n,JX_n,JX_1)$.  Using \eqref{brackets-centro-2}, we get
$$\big[[X_n,JX_n],JX_1\big] \ = \
2\,\sum\nolimits_{i=1}^{n-2}\left(b_{n-1}^1\,\gamma_i+B_{n-1}^1\,\alpha_i\right)\,X_i.$$
At this point, we ignore the value of $[X_n,JX_n]$. However, it is clear that such bracket will depend, at most,
on all the other elements of the basis, $\{X_i,JX_{i}\}_{i=1}^{n-1}$.
Since the bracket of each one of these elements with $JX_1$ equals zero, as \eqref{brackets-centro-2} shows, we can conclude that $\big[[X_n,JX_n],JX_1\big]=0$. Hence,
\begin{equation}\label{ec1}
b_{n-1}^1\,\gamma_i+B_{n-1}^1\,\alpha_i = 0, \ \forall\, i=1,\ldots,n-2.
\end{equation}
We will make use of these equations below.

Let us now go back to the series~\eqref{serie1}. To finish the proof we will study two different cases, namely,
dim\,$(\frg_3\cap J\frg_1)\geq 2$ and dim\,$(\frg_3\cap J\frg_1)=1$.
We will see that none of them is valid and in this way, condition $\frg_3\cap J\frg_1\neq\{0\}$ will be contradicted.  Since we have already refused the condition $\frg_3\cap J\frg_1=\{0\}$ (see p. \pageref{contra}) the proof of the theorem will be completed.


\medskip\noindent
$\bullet$ \ \textsc{Study of} dim\,$(\frg_3\cap J\frg_1)\geq 2$.

\smallskip
\noindent Assume there is another element in $\frg_1$ whose image by $J$ belongs to~$\frg_3$.
Arranging generators if necessary, we can suppose $JX_2\in\frg_3$:
$$\frg_1=\langle X_1,\ldots,X_{n-2}\rangle, \quad \frg_2=\langle X_1,\ldots,X_{n-1},JX_{n-1}\rangle,\quad \frg_3\supseteq\langle X_1,\ldots,X_{n-1},JX_1,JX_{2},JX_{n-1}\rangle.$$
Then, it should be possible to find an element $Y\in\frg$ such that $[JX_2,Y]\in\frg_2$ and $[JX_2,Y]\notin\frg_1$.
From the brackets \eqref{brackets-centro-2}, one can see that
$Y$ must depend on $X_n$ or $JX_n$. Therefore, repeating a similar argument to that for $JX_1$
we have
$$
[X_n,JX_2] =b_{n-1}^2\,X_{n-1}+B_{n-1}^2\,JX_{n-1},\quad
[JX_2,JX_n] =B_{n-1}^2\,X_{n-1}-b_{n-1}^2\,JX_{n-1},
$$
where $b_{n-1}^{2},B_{n-1}^{2}\in\mathbb{R}$. In particular
$(b_{n-1}^2,B_{n-1}^2)\neq(0,0)$ since $JX_{2}$ does not belong to $\frg_{2}$.
The Jacobi identity \eqref{Jacobi} for the triplet $(X_n, JX_n, JX_2)$ together with \eqref{brackets-centro-2} leads to
$$\big[[X_n,JX_n],JX_2\big] \ = \ 2\,\sum\nolimits_{i=1}^{n-2}\left(b_{n-1}^2\,\gamma_i+B_{n-1}^2\,\alpha_i\right)\,X_i.$$
By an analogous reasoning to that used for \eqref{ec1}, we get
\begin{equation}\label{ecuation2}
b_{n-1}^2\,\gamma_i+B_{n-1}^2\,\alpha_i = 0, \ \forall\, i=1,\ldots,n-2.
\end{equation}
Let us solve the system of equations given by \eqref{ec1} and \eqref{ecuation2}.

We first suppose  that $b_{n-1}^1=0$.  Then $B_{n-1}^1\neq 0$ and we get  $\alpha_i=0$ for $1\leq i \leq n-2$,
from~\eqref{ec1}. Observe that
$(\gamma_1,\ldots,$ $\gamma_{n-2})\neq(0,\ldots,0)$, or otherwise $X_{n-1},$ $JX_{n-1}\in\frg_1$, which
is a contradiction. Therefore,
from \eqref{ecuation2} we conclude $b_{n-1}^2=0$. Now, we take
$0\neq X=B_{n-1}^2\,JX_1-B_{n-1}^1\,JX_2\ \in \ J\frg_1.$
Note that the bracket of $X$ with every element of the basis vanishes. Thus $X\in\frg_1\cap J\frg_1$, but this is not possible as
$J$ is of strongly non-nilpotent type.

Suppose now $b_{n-1}^1\neq 0$.  Then we can solve $\gamma_i$ from \eqref{ec1}
and replacing these values in \eqref{ecuation2}, we get:
$$\alpha_i\,\left(B_{n-1}^2-\frac{B_{n-1}^1\,b_{n-1}^2}{b_{n-1}^1}\right)=0,  \quad 1\leq i \leq n-2.$$
Notice that there exists some $\alpha_i\neq 0$ (otherwise, $\gamma_i=0$, for every $1\leq i \leq n-2$), so the other factor in the equations above should vanish. In particular, this implies that  $b_{n-1}^2\neq 0$ (or otherwise $(b^2_{n-1},B^2_{n-1})=(0,0)$, which is not allowed).
Take $0\neq X=b_{n-1}^2\,JX_1-b_{n-1}^1\,JX_2$. Again, $X\in \frg_1\cap J\frg_1$, and a contradiction is attained.

This concludes the analysis of this case and shows that it is not a valid one.

\medskip

\noindent
$\bullet$ \ \textsc{Study of} dim\,$(\frg_3\cap J\frg_1)=1$.

\smallskip\noindent
Under this hypothesis, we have that $JX_k\notin\frg_3$ for every $2\leq k \leq n-2$. We next divide the study according to the dimension of $\frg_3$.  By \eqref{serie1}, simply recall that dim\,$\frg_3\geq n+1$.

\smallskip
\begin{addmargin}[1.5em]{0em}
$\sqbullet$ \ Let dim\,$\frg_3> n+1$.
In this case, it is clear that $X_n\in\frg_3$ so it is possible to fix the unknown real brackets in
\eqref{brackets-centro-2}:
$$\begin{array}{rcl}
[X_n,JX_k] &=& \sum_{i=1}^{n-2}c_i^k\,X_i+c_{n-1}^k\,X_{n-1}+C_{n-1}^k\,JX_{n-1}, \qquad k=2,\ldots,n-2,\,n, \\[4pt]
[JX_k,JX_n] &=& -\sum_{i=1}^{n-2}c_i^k\,JX_i+C_{n-1}^k\,X_{n-1}-c_{n-1}^k\,JX_{n-1}, \qquad k=2,\ldots,n-2.
\end{array}$$
If there is some $k\in\{2,\ldots,n-2\}$ such that $c_i^k=0$, for every $1\leq i \leq n-2$,
then $JX_k\in\frg_3$. However, this is not possible by the hypothesis of this case. Thus, we can ensure
$(c_1^k,\ldots,c_{n-2}^k)\neq (0,\ldots,0)$, for every $2\leq k \leq n-2$. Let us note that this
implies $JX_n\notin\frg_3$. Hence,
$$\frg_1=\langle X_1,\ldots,X_{n-2}\rangle, \quad
  \frg_2=\langle X_1,\ldots,X_{n-1},JX_{n-1}\rangle, \quad
  \frg_3=\langle X_1,\ldots,X_{n},JX_1,JX_{n-1}\rangle,$$
so in particular, dim\,$\frg_3 = n+2$. We move to the study $\frg_4$ for this case.

Due to the nilpotency of the Lie algebra, at least $JX_2$ (up to arrangement of generators) or $JX_n$ should belong to $\frg_4$.
If $JX_2\in\frg_4$, then in particular $[JX_2,JX_n]\in\frg_3$, so $c_i^2=0$, for $2\leq i\leq n-2$.
The Jacobi identity $Jac\,(X_n,JX_n,JX_2)$ implies
$$0 \ = \
-c_1^2\,(b_{n-1}^1\,X_{n-1}+\,B_{n-1}^1\,JX_{n-1})-2\,
\sum\nolimits_{i=1}^{n-2}\left(c_{n-1}^2\,\gamma_i+C_{n-1}^2\,\alpha_i\right)X_i.
$$
Since $(b_{n-1}^1,B_{n-1}^1)\neq (0,0)$, necessarily $c_1^2=0$. Nonetheless, this yields $JX_2\in\frg_3$, which is not allowed. As a consequence, we are led to consider
$$\frg_4 = \langle X_1,\ldots,X_{n},JX_1,JX_{n-1}, JX_n\rangle.$$ However, as $[JX_k,JX_n]\in\frg_3$, one has that $c_i^k=0$, for every $2\leq i,k \leq n-2$, but this implies that  $JX_k\in\frg_4$, for $2\leq k \leq n-2$, which is a contradiction.
Hence, dim\,$\frg_3$ cannot be greater than $n+1$.


\medskip\noindent
$\sqbullet$ \ Let dim~$\frg_3 = n+1$, i.e.,
$$\frg_1=\langle X_1,\ldots,X_{n-2}\rangle, \quad
  \frg_2=\langle X_1,\ldots,X_{n-1},JX_{n-1}\rangle, \quad
  \frg_3=\langle X_1,\ldots,X_{n-1},JX_1,JX_{n-1}\rangle.$$
We will see that this is also an invalid case by carefully studying $\frg_4$.
Notice that due to the construction of the ascending central series, $\frg_4\cap J\frg_1\neq \{0\}$. We are going to consider two different cases according to the dimension of $\frg_4\cap J\frg_1$.

\smallskip
\begin{addmargin}[1.5em]{0em}
$\triangleright$ \
Assume $\dim (\frg_{4}\cap J\frg_{1})=1$. Since $JX_1\in\frg_3\subset\frg_4$, this implies that
$JX_k\notin\frg_4$, for every $2\leq k \leq n-2$. Then, $\frg_4\cap \langle X_n, JX_n\rangle\neq\{0\}$.
Observe that, without loss of generality, we can suppose $X_n\in\frg_4$ because the role of the elements $X_n$ and $JX_n$
in \eqref{brackets-centro-2} is interchangeable.
Moreover, it is easy to see that $JX_n\notin\frg_4$ (otherwise, $\frg_4=\frg$ which is a contradiction) and thus
$$\frg_4=\langle X_1,\ldots,X_{n},JX_1,JX_{n-1}\rangle.$$
Recall that the nilpotency of the algebra implies that the ascending central series finishes in $\frg$,
so there is $Y\in\frg_5$ such that $Y\notin\frg_4$. Necessarily, $Y$ is linearly dependent
with some $JX_2,\ldots,$ $JX_{n-2},$ $JX_n$. In fact, we can assume that, up to a change of
basis, $Y=JX$ for certain $X\in \mathcal W=\langle X_1,\ldots,X_n\rangle\subset\frg_4$.
Let us note that we are in the conditions of Lemma \ref{lemma1_serie_asc} \textrm{(ii)} with $r=4$.
Thus, there exists $Z\in\frg_3$ such that $JZ\in\frg_4$ and $JZ\notin\frg_3$. However, if we take
an element $Z=\sum_{i=1}^{n-1}s_i X_i+S_1 JX_1+S_{n-1} JX_{n-1}\in\frg_3$, then we can
see that $JZ\in\frg_4$ if and only if $JZ\in\frg_3$, which is a contradiction.

\smallskip\noindent
$\triangleright$ \ We next study $\dim (\frg_{4}\cap J\frg_{1})>1$.
Up to arrangement
of generators one can assume that $JX_2\in\frg_4$, so
$$\frg_4 \supseteq \langle X_1,\ldots,X_{n-1},JX_1,JX_2,JX_{n-1}\rangle.$$
Then, there should be an element $Y\in\frg$ such that $[JX_2,Y]\in\frg_3$ and $[JX_2,Y]\notin\frg_2$. Necessarily, $Y$ depends
on $X_n$ or $JX_n$ and it is possible to set the real brackets
$$\begin{array}{ccc}
[X_n,JX_2] &=& \sum_{i=1}^{n-2} c_i^2\,X_i+c_{n-1}^2\,X_{n-1}+C_1^2\,JX_1+C_{n-1}^2\,JX_{n-1},\\[4pt]
[JX_2,JX_n] &=& -\sum_{i=1}^{n-2} c_i^2\,JX_i-c_{n-1}^2\,JX_{n-1}+C_1^2\,X_1+C_{n-1}^2\,X_{n-1}.
\end{array}$$
Observe that $c_i^2=0$, for every $2\leq i \leq n-2$, because both brackets belong to $\frg_3$. Furthermore,
they cannot lie in $\frg_2$, so one has $(c_1^2,C_1^2)\neq (0,0)$. Now, imposing $Jac\,(X_n,JX_n,JX_2)$ the following system of equations must be solved:
$$\left\{\begin{array}{ccc}
0 &=& C_1^2\,B_{n-1}^1+c_1^2\,b_{n-1}^1,\\[3pt]
0 &=& C_1^2\,b_{n-1}^1-c_1^2\,B_{n-1}^1.
\end{array}\right.$$
Since $(c_1^2,C_1^2)\neq (0,0)$, and  $(b_{n-1}^1,B_{n-1}^1)\neq (0,0)$, the system has no solution.
\end{addmargin}
\medskip
This finishes the study of this case and the proof of Theorem~\ref{prop_centro}.
\end{addmargin}\end{proof}

\subsection{Some consequences and applications}\label{aplications}

In this section we point out some consequences of the results obtained in the previous sections, as well as
some applications to the study of existence of complex structures on certain products of nilpotent Lie algebras.

In six dimensions, if an NLA $\frg$ has a strongly non-nilpotent complex structure then its center has dimension 1
(see Theorem~\ref{structural-dim6}).
Alternatively, this result can also be obtained as a consequence of Proposition~\ref{dim-g} \textrm{(i)}.
In dimension 8, by Theorem~\ref{prop_centro} we directly have:

\begin{corollary}\label{corolario_centro_SnN}
Nilpotent Lie algebras of dimension eight admitting strongly non-nilpotent complex structures have
$1$-dimensional centers.
\end{corollary}

Note that in dimension 10, there are NLAs with 2-dimensional centers admitting SnN complex structures,
as Example~\ref{ejemplo10} illustrates.
Thus, in ten dimensions the upper bound in Theorem~\ref{prop_centro} is attained.

In \cite[Corollary 7]{U} it is proved that a given $6$-dimensional NLA can only admit
complex structures of either nilpotent or non-nilpotent type (recall Figure~\ref{diagrama_tipos_J}).
Moreover, in the second case the structure is necessarily SnN.
The next result shows that something similar happens in eight dimensions, in the sense that quasi-nilpotent
and strongly non-nilpotent complex structures cannot coexist on the same NLA.
Notice that a similar statement does not hold in dimension 10, as Example~\ref{ejemplo10} illustrates.

\begin{corollary}\label{corollary-estructura-acs}
Let $\frg$ be an NLA of dimension $8$ admitting complex structures. Then, all of them are either quasi-nilpotent
or strongly non-nilpotent.
\end{corollary}

\begin{proof}
On the one hand, if $\frg$ admits an SnN complex structure~$J$, then by Theorem~\ref{prop_centro}
the center~$\frg_1$ has dimension~1. On the other hand, if~$J'$ is a quasi-nilpotent complex structure on~$\frg$, then by Definition~\ref{tipos_J} we have
$\{0\}\neq\fra_1(J')\subseteq \frg_1$. Since~$\fra_1(J')$ has dimension at least~2, the
center~$\frg_1$ must have dimension~$\geq 2$. This is a contradiction,
so~$J$ and~$J'$ cannot coexist on~$\frg$.
\end{proof}

It is a difficult problem to decide if complex structures exist or not on a given product Lie algebra (see for example the
recent paper \cite{CS} about complex structures on six-dimensional
product Lie algebras).
As a consequence of Theorem~\ref{prop_centro} we have the following result for complex structures on certain products of NLAs.

\begin{corollary}\label{product-1}
Let $\frg$ be an NLA of dimension $2n$, with $n \geq 4$, admitting a complex structure $J$.
If~$\frg$ is the product of $n-2$ Lie algebras, i.e. $\frg\cong \frh_1\times\cdots\times\frh_{n-2}$, then
$J$ is quasi-nilpotent.
\end{corollary}

\begin{proof}
Since the center of $\frh_j$, $1 \leq j \leq n-2$, has dimension at least 1, we have that $\dim \frg_1 \geq n-2$.
By Theorem~\ref{prop_centro}, we know that if $J$ is SnN then $\dim \frg_1 \leq n-3$.
Consequently, $J$ can only be of quasi-nilpotent type.
\end{proof}

The previous corollary for $n=4$ gives a restriction on the type of complex structures that an 8-dimensional product nilpotent
Lie algebra can admit. In the following examples we illustrate how to make use of this result to study the existence of
complex structures on product Lie algebras.

\begin{example}\label{ex-product-1}
Let us consider the 8-dimensional Lie algebra $\frg=\frh\times \mathbb{R}$,
where $\frh$ is the nilpotent Lie algebra generated by $\{X_1,\ldots,X_7\}$ satisfying
$$
[X_1,X_2]=X_3, \ [X_1,X_3]=X_4, \ [X_1,X_4]=X_5, \ [X_2,X_3]=X_6, \ [X_1,X_5]=[X_2,X_6]=X_7.
$$
Let us denote by $X_8$ a generator of $\mathbb{R}$. It is easy to see that the NLA $\frg$
has ascending type $(2,4,5,6,8)$.

Suppose that $\frg$ admits a complex structure $J$.
By Corollary~\ref{product-1} we have that $J$ must be quasi-nilpotent.
Moreover, since the center of $\frg$ is given by
$$
\frg_1= \langle X_7,X_8 \rangle,
$$
it is clear that the first term in the ascending series of $\frg$ compatible with $J$ is $\fra_1(J)=\frg_1$.
Now, the argument at the end of Section~\ref{complex-structures}
implies that the quotient Lie algebra $\widetilde{\frg}_1=\frg/\fra_1(J)=\frg/\langle X_7,X_8 \rangle$ has a complex structure $\widetilde{J}_1$.
However, $\widetilde{\frg}_1$ is the 6-dimensional NLA generated by
$\{\widetilde{X}_1,\ldots,\widetilde{X}_6\}$ with
$$
[\widetilde{X}_1,\widetilde{X}_2]=\widetilde{X}_3, \ \ [\widetilde{X}_1,\widetilde{X}_3]=\widetilde{X}_4, \ \
[\widetilde{X}_1,\widetilde{X}_4]=\widetilde{X}_5, \ \ [\widetilde{X}_2,\widetilde{X}_3]=\widetilde{X}_6,
$$
and this NLA does not admit any complex structure, as proved in \cite{S}.

In conclusion, the 8-dimensional NLA $\frg=\frh\times \mathbb{R}$ cannot admit any complex structure.
\end{example}

\begin{example}\label{ex-product-2}
Let us consider the 8-dimensional Lie algebra $\frg=\frh_1\times \frh_2$,
where $\frh_1$ is the nilpotent Lie algebra of dimension 5 generated by $\{X_1,\ldots,X_5\}$ satisfying
$$
[X_1,X_2]=X_4, \ \ \ [X_1,X_4]=[X_2,X_3]=X_5,
$$
and $\frh_2$ is the nilpotent Lie algebra of dimension 3 given by $\{X_6,X_7,X_8\}$ with
$[X_6,X_7]=X_8$.

It is easy to see that $\frg$ has ascending type $(2,6,8)$ and the center of $\frg$ is exactly
$$
\frg_1= \langle X_5,X_8 \rangle.
$$
If the NLA $\frg$ admits a complex structure $J$, then
$J$ is necessarily quasi-nilpotent by Corollary~\ref{product-1}.
Moreover, since $\dim\frg_1=2$ we have that the first term in the ascending series of $\frg$ compatible with $J$ is $\fra_1(J)=\frg_1$.

By the last part of Section~\ref{complex-structures},
the quotient Lie algebra $\widetilde{\frg}_1=\frg/\fra_1(J)=\frg/\langle X_5,X_8 \rangle$
has a complex structure $\widetilde{J}_1$. Let
us observe that the 6-dimensional NLA $\widetilde{\frg}_1$ is generated by
$\{\widetilde{X}_1,\widetilde{X}_2,\widetilde{X}_3,\widetilde{X}_4,\widetilde{X}_6,\widetilde{X}_7\}$ satisfying
$$
[\widetilde{X}_1,\widetilde{X}_2]=\widetilde{X}_4.
$$
This Lie algebra is precisely
the product of the 3-dimensional Heisenberg Lie algebra and $\mathbb{R}^3$, so it has complex structures.
Indeed, it follows from \cite{COUV} that, up to isomorphism, the complex structure $\widetilde{J}_1$
on $\widetilde{\frg}_1$ is determined by
$$
\widetilde{J}_1 \widetilde{X}_1=-\widetilde{X}_2, \quad \widetilde{J}_1 \widetilde{X}_3=\widetilde{X}_4,
\quad \widetilde{J}_1 \widetilde{X}_6=-\widetilde{X}_7.
$$
Now, we can extend this complex structure to $\frg$ as follows. Define the almost complex structure $J$ on $\frg$ as
$$
J X_1=-X_2, \quad J X_3=X_4, \quad J X_6=-X_7, \quad J X_5=-X_8.
$$
A direct calculation shows that $J$ is integrable. Hence, the product Lie algebra $\frg=\frh_1\times \frh_2$ has a complex structure.
\end{example}

\section{The 8-dimensional case}\label{sec:8-dim}
\noindent
In this section we determine the structure of the ascending central series
for those 8-dimensional NLAs admitting strongly non-nilpotent complex structures.
More concretely, we prove the following structural theorem,
which is the analogous result in eight dimensions to Theorem~\ref{structural-dim6}.

\begin{theorem}\label{teorema-estructura-acs}
Let $\frg$ be an NLA of dimension $8$. If $\frg$ admits an SnN complex structure, then $\frg$ has ascending type
$(1,3,8)$, $(1,3,5,8)$, $(1,3,6,8)$, $(1,3,5,6,8)$, $(1,4,8)$, $(1,4,6,8)$, $(1,5,8)$, or $(1,5,6,8)$.
\end{theorem}

For the proof of this theorem, in Section~\ref{sec:g2} we start studying the term $\frg_2$
of the ascending central series $\{\frg_k\}_k$,
and we devote Section~\ref{sec:g34} to the remaining terms $\frg_k$, $k \geq 3$.
It is important to remark that,
since our method is constructive, we will be able to describe the complex structure equations that parametrize
all the SnN complex structures in dimension~8. This is detailed in the last Section~\ref{SnN-equations-dim8}.

\smallskip
Before beginning with the proof, we would like to
explain the main ideas on which the method is based.

\smallskip
We first observe that there are some general results in dimension $8$ that come
straightforward from Section~\ref{SnN}. More precisely, if
$\frg$ is an $8$-dimensional NLA with an SnN complex structure $J$, then:
\begin{equation}\label{resultados-iniciales}
\begin{cases}
\bullet \ \dim \frg_1=1 \text{ (see Theorem~\ref{prop_centro} and Corollary \ref{corolario_centro_SnN})}; \\
\bullet \ 2 \leq \dim \frg_2 \leq 5 \text{ (see Proposition \ref{dim-g}, (ii))}; \\
\bullet \ \frg_2 \cap J\frg_1=\{0\} \text{ (see Lemma \ref{lemma1_SnN})}; \\
\bullet \ \frg_3\cap J\frg_3\neq \{0\} \text{ and } \dim\frg_3\geq 4 \text{ (see Proposition \ref{dim-g}, (iii) with $k=1$, $r=3$)}.
\end{cases}
\end{equation}

\noindent Furthermore, thanks to Garc\'ia Vergnolle and Remm \cite{GVR},
we know that if an 8-dimensional NLA is quasi-filiform, i.e., its nilpotency step is 6, then it
does not admit any complex structure. This fact together
with Corollary~\ref{corolario_smayor3} allows us to conclude that an 8-dimensional NLA $\frg$
endowed with a strongly non-nilpotent complex structure has nilpotency step $3\leq s\leq 5$.
Consequently, $\dim \frg_1=1$ and $\frg_5=\frg$, so the study of the ascending central series
is reduced to the terms $\frg_2$, $\frg_3$, and $\frg_4$.
For this aim we will use the same ideas contained in the proof of Theorem~\ref{prop_centro}. That is,
we will construct a $J$-adapted basis $\mathcal B$ of~$\frg$ that will be also adapted to the nilpotent structure of $\frg$.
Let us begin the construction.

\smallskip
Since $\dim \frg_1=1$, we can establish
$\frg_1=\langle X_1\rangle$.
In this way, $X_1$ and $JX_1$ can be considered two elements in the $J$-adapted basis $\mathcal{B}$.
Moreover, it is clear that
$[X_1,\cdot\,]=0$ and $[JX_1,J\cdot \,]=J[JX_1,\cdot\,]$, as $X_1$ is a central element and $J$ is integrable, so
it satisfies the Nijenhuis condition \eqref{Nijenhuis}.

We now focus on $\frg_2$. By the nilpotency of $\frg$ and conditions \eqref{resultados-iniciales}, there is a vector
$Y\in\frg_2$ linearly independent with $X_1$ and $JX_1$.
Let us denote $X_{2}=Y$. Then,
$$\frg_1=\langle X_1\rangle,\quad \frg_2\supseteq \langle X_1, X_2\rangle.$$
Observe that $X_2$ and its image by $J$ can be taken as new elements of $\mathcal{B}$.  By construction, $[X_2,\,\cdot\,]\in\langle X_1\rangle$. Furthermore,
$[X_2,JX_1]=\lambda\,X_1$ and thus $[JX_1,JX_2]=-\lambda\,JX_1$, for
some $\lambda\in\mathbb{R}$, due to \eqref{Nijenhuis}. The nilpotency of the Lie algebra requires $\lambda=0$.
Hence,
\begin{equation}\label{inicio-brackets}
[X_1,\cdot \,]=0,\qquad  [X_2,JX_1]=0,\qquad [X_2,JX_2]=b_{22}^1X_1,\qquad [JX_1,JX_2]=0,
\end{equation}
where $b_{22}^1\in\mathbb R$.  This is the starting point of our construction.
We also introduce the following notation, that will be of great help in the next sections:

\begin{notation}\label{notation}
{\rm Let $\{X_i,JX_i\}_{i=1}^n$ be a $J$-adapted basis of $\frg$.
For each $1\leq i\leq n$, we denote $\mathcal V_i = \langle X_i, JX_i\rangle$.  In this way, as a vector space,~$\frg$ can be decomposed as direct sum of the $2$-dimensional $J$-invariant spaces~$\mathcal V_i$, that is,
$\frg=\bigoplus_{i=1}^n \mathcal V_n$.}
\end{notation}

\subsection{The term $\frg_2$}\label{sec:g2}
\noindent
In this section we provide the possible dimensions for the term $\frg_2$. Moreover, we give a description of this space
in terms of elements of the $J$-adapted basis $\mathcal B$.


\begin{proposition}\label{prop_g2_J}
Let $(\frg,J)$ be an $8$-dimensional NLA endowed with a strongly non-nilpotent complex structure.
Then, $\frg_{2}\cap J\frg_{2}\neq \{0\}$. As a consequence, $3\leq $ dim\,$\frg_2\leq 5$.
\end{proposition}

\begin{proof}
We proceed by contradiction.
Let us suppose that $\frg_{2}\cap J\frg_{2}=\{0\}$.
From Proposition~\ref{dim-g}, part~\textrm{(i)} we have
$\dim \frg_{2}=2$, thus
$$\frg_{1}=\langle X_{1}\rangle, \quad \frg_{2}=\langle X_{1},X_{2}\rangle.$$
Moreover, by part \textrm{(i)} we know that
the nilpotency step $s$ of $\frg$ satisfies $s=4$ or $s=5$.

Let us now focus on~$\frg_3$.
By \eqref{resultados-iniciales} we know that $\frg_{3}\cap J\frg_{3}\neq\{0\}$.
As a consequence of Lemma~\ref{lemma1_SnN}~\textrm{(i)} with $k=2$,
we also have that
$\frg_3\cap J\frg_2=\{0\}$, so
we can ensure the existence of a $2$-dimensional subspace $\mathcal V_3=\langle X_3,JX_3\rangle$ contained in~$\frg_3$.
Furthermore, due to Proposition~\ref{corolario_step} \textrm{(ii)}
it is not possible to find another element in~$\frg_3$
linearly independent with
$\{X_{i}, JX_{i}\}_{i=1}^3$.
Therefore, the ascending central series comes as follows:
$$\frg_1=\langle X_1\rangle,\quad \frg_2=\langle X_1,X_2\rangle,\quad
    \frg_3=\langle X_1,X_2,X_{3},JX_{3}\rangle.$$
Observe that one can find a vector in $\frg$, that we will call $X_4$, such that
$\mathcal B=\{X_i, JX_i\}_{i=1}^4$ and $X_4, JX_4\notin\frg_3$.
With this information, it is possible to add some new Lie brackets to those in \eqref{inicio-brackets}:
$$[X_2, X_k]_{k=3,4} = a_{2k}^1 X_1,\qquad [X_2, JX_k]_{k=3,4} = b_{2k}^1 X_1,\qquad
[X_3, X_4] = a_{34}^1 X_1+a_{34}^2 X_2,$$
$$[X_3, JX_k]_{1\leq k \leq 4} = b_{3k}^1 X_1 + b_{3k}^2 X_2,\qquad [X_4, JX_3] = b_{43}^1 X_1 + b_{43}^2 X_2,$$
where the coefficients are real numbers.  Furthermore, $$[JX_1, JX_3] = J[JX_1, X_3] = -b_{31}^1JX_1 - b_{31}^2JX_2\in\frg_2,$$ so we conclude that $b_{31}^1=b_{31}^2 = 0$.  In a similar way, imposing the equations $N_J(X_2, X_3)$ and $N_J(X_3, X_4)$ given by \eqref{Nijenhuis}, we obtain
\begin{equation*}
\begin{array}{ll}
b_{32}^1 = b_{23}^1,\quad b_{32}^2=0,&\quad [JX_2, JX_3] = a_{23}^1 X_1, \\[3pt]
b_{43}^1 = b_{34}^1,\quad b_{43}^2 = b_{34}^2,&\quad  [JX_3, JX_4] = a_{34}^1 X_1 + a_{34}^2 X_2.
\end{array}
\end{equation*}
Up to this moment, we have the following real brackets:
\begin{equation}\label{bracket_22C}
\begin{array}{ll}
[X_1,\cdot\ ]=0,
  & [X_4, JX_k]_{k=1,2} \text{  unknown}, \\[2pt]
[X_2,X_k]_{k=3,4}=a_{2k}^1\,X_1,
  & [X_{4},JX_{3}]=b_{34}^1\,X_{1}+b_{34}^2\,X_{2}, \\[2pt]
[X_2,JX_1]=0,
  & [X_4, JX_4] \text{ unknown}, \\[2pt]
[X_2,JX_k]_{k=2,3,4}=b_{2k}^1\,X_1,
  & [JX_1,JX_k]_{k=2,3}=0, \\[2pt]
[X_{3},X_{4}]=a_{34}^1\,X_{1}+a_{34}^2\,X_{2},
  & [JX_1, JX_4] \text{ unknown}, \\[2pt]
[X_{3},JX_{1}]=0,
  & [JX_{2},JX_{3}]=a_{23}^1\,X_{1}, \\[2pt]
[X_{3},JX_{2}]=b_{23}^1\,X_{1},
  & [JX_2, JX_4] \text{ unknown}, \\[2pt]
[X_{3},JX_{k}]_{k=3,4}=b_{3k}^1\,X_{1}+b_{3k}^2\,X_{2}, \quad
  & [JX_{3},JX_{4}]=a_{34}^1\,X_{1}+a_{34}^2\,X_{2}.
\end{array}
\end{equation}
Moreover, for those brackets above involving $X_4$, we can write:
$$
[X_4, JX_k] = \sum\nolimits_{i=1}^3 b_{4k}^iX_i + \sum\nolimits_{j=1, j\neq k}^4 c_{4k}^j JX_j,\qquad k=1,2,4.
$$
Now, using the Nijenhuis condition and the nilpotency of $\frg$, we have:
\begin{equation*}
\begin{array}{ll}
b_{41}^1=0,&\qquad [JX_1, JX_4] = -\sum_{i=2}^3 b_{41}^i JX_i + \sum_{j=2}^4 c_{41}^j X_j,\\[3pt]
b_{42}^2=0, & \qquad [JX_2, JX_4] =  (a_{24}^1 + c_{42}^1) X_1 + c_{42}^3 X_3 + c_{42}^4 X_4 + (b_{24}^1 - b_{42}^1) JX_1 - b_{42}^3 JX_3.
\end{array}
\end{equation*}
Hence, the unknown brackets in~\eqref{bracket_22C} can be expressed as:
\begin{equation}\label{bracket_22C-2}
\begin{array}{l}
[X_4, JX_1] = \sum_{i=2}^3 b_{41}^iX_i + \sum_{j=2}^4 c_{41}^j JX_j,\qquad \qquad [X_4, JX_4] = \sum_{i=1}^3 b_{44}^iX_i + \sum_{j=1}^3 c_{44}^j JX_j,\\[4pt]
[X_4, JX_2] = \sum_{i=1, i\neq 2}^3 b_{42}^iX_i + \sum_{j=1, j\neq 2}^4 c_{42}^j JX_j,\quad [JX_1, JX_4] = -\sum_{i=2}^3 b_{41}^i JX_i + \sum_{j=2}^4 c_{41}^j X_j,\\[4pt]
[JX_2, JX_4] =  (a_{24}^1 + c_{42}^1) X_1 + c_{42}^3 X_3 + c_{42}^4 X_4 + (b_{24}^1 - b_{42}^1) JX_1 - b_{42}^3 JX_3.
\end{array}
\end{equation}
Observe that the position of $JX_k$ ($k=1,2,4$) in the ascending central series will immediately
make some of the coefficients above
equal to zero. Indeed, recall that $\frg$ can be either 4-step or 5-step nilpotent. We will next analyze these two possibilities.

\smallskip\noindent
$\bullet$ \ \textsc{The $4$-step case}.

\smallskip\noindent
To ensure that $\frg_4=\frg$, the following values in~\eqref{bracket_22C-2} are needed:
$$b_{24}^1 = b_{42}^1,\quad b_{41}^2 = c_{41}^2=c_{41}^4=c_{42}^1=c_{42}^4=c_{44}^1=c_{44}^2=0.$$  Note that the only (possibly) non-zero brackets involving $JX_1$ are $[X_4, JX_1] = b_{41}^3 X_3 + c_{41}^3 JX_3$
and $[JX_1,JX_4]=c_{41}^3\,X_3-b_{41}^3\,JX_3$.  Since $JX_1\notin \frg_1$, necessarily $(b_{41}^3, c_{41}^3)\neq (0,0)$.

Let us study the Jacobi identity to contradict this case. Considering $Jac(X_3,JX_1,Y)$, for $Y=X_4, JX_4$, we have:
$$
\left.\begin{array}{r}
c_{41}^3\big(b_{33}^1\,X_{1}+b_{33}^2\,X_{2}\big)=0,\\[2pt]
b_{41}^3\big(b_{33}^1\,X_{1}+b_{33}^2\,X_{2}\big)=0,
\end{array}\right\}
\text{ so we conclude } b^1_{33}=b^2_{33}=0.
$$
From $Jac\,(X_{2},JX_{1},Y)$, with $Y=X_4, JX_4$, we obtain
$$
\begin{pmatrix}
a^1_{23} & b^1_{23} \\
b^1_{23} & - a^1_{23}
\end{pmatrix}
\begin{pmatrix} b^3_{41} \\ c^3_{41} \end{pmatrix}
=
\begin{pmatrix} 0 \\ 0 \end{pmatrix},
\text{ and thus } a^1_{23}=b^1_{23}=0.
$$
Now, we observe that due to the previous choices we need $(a^2_{34}, b^2_{34})\neq (0,0)$, or otherwise $X_3\in\frg_2$.
If we compute $Jac(X_3,JX_2,Y)$, for $Y=X_4,JX_4$, we get
$$b^1_{22}\,a^2_{34} = b^1_{22}\,b^2_{34} = 0, \text{ which implies } b^1_{22}=0.$$
Finally, using $Jac(X_4,JX_4, Y)$ with $Y=X_3,JX_3$ one has:
$$
\begin{pmatrix}
b^1_{24} & -a^1_{24} \\
a^1_{24} & b^1_{24}
\end{pmatrix}
\begin{pmatrix} a^2_{34} \\ b^2_{34} \end{pmatrix}
=
\begin{pmatrix} 0 \\ 0 \end{pmatrix},
\text{ and necessarily } a^1_{24}=b^1_{24}=0.
$$
Nonetheless, replacing all these zero values in \eqref{bracket_22C}, one then realizes that $X_2$ belongs to the center of~$\frg$, which is a contradiction. Therefore,
this case does not hold and the algebra must be $5$-step nilpotent.

\smallskip\noindent
$\bullet$ \ \textsc{The $5$-step case}.

\smallskip\noindent
First, we observe that $\frg_4\cap J\frg_2\neq \{0\}$.  In fact, if we assume $\frg_4\cap J\frg_2=\{0\}$, then  $X_4\in\frg_4$ and we are in the conditions of Lemma~\ref{lemma1_serie_asc} \textrm{(ii)} for $r=4$.
This means that there should be a vector $W\in\frg_3$ such that $JW\in\frg_4$ but $JW\notin\frg_3$.  Due to the arrangement of the ascending central series,  the only possibility is having $W\in\frg_2$, which contradicts our assumption.

In addition, we remark that the nilpotency of $\frg$ yields dim\,$\frg_4 = 5$ or $6$.  We will next
see that none of these cases is valid, thus finishing the proof.

\smallskip\noindent
\begin{addmargin}[1.5em]{0em}
$\sqbullet$ \ Let dim\,$\frg_4 = 5$.  Then,
$$\frg_4=\langle X_1, X_2, X_3, JX_3, JY\rangle,\text{ where } Y\in\frg_2.$$
Let us express $Y=\alpha_1 X_1 + \alpha_2 X_2.$ After a suitable arrangement of generators, we can assume that either $(\alpha_1, \alpha_2) = (1,0)$ or $(\alpha_1, \alpha_2) = (0,1)$.  We analize these two possibilities in the next lines.

On the one hand, if $Y=X_1$ then
$$\frg_1=\langle X_1\rangle,\,\, \frg_2=\langle X_1,X_2\rangle,\,\,
    \frg_3=\langle X_1,X_2,X_{3},JX_{3}\rangle,\,\, \frg_4=\langle X_1,X_2,X_{3},JX_{1},JX_{3}\rangle, \,\, \frg_{5}=\frg.$$
This structure of the ascending central series forces the Lie brackets to follow~\eqref{bracket_22C} and~\eqref{bracket_22C-2} with $$b_{41}^2 = c_{41}^2 = c_{41}^4 = c_{42}^4  = c_{44}^2=0.$$
The contradiction is reached repeating the same calculations for the Jacobi identity performed in the case $\frg_4 = \frg$.

On the other hand, if
$$\frg_1=\langle X_1\rangle,\,\, \frg_2=\langle X_1,X_2\rangle,\,\,
    \frg_3=\langle X_1,X_2,X_{3},JX_{3}\rangle,\,\, \frg_4=\langle X_1,X_2,X_{3},JX_{2},JX_{3}\rangle, \,\, \frg_{5}=\frg,$$
a similar argument can be applied to obtain a contradiction after studying the Jacobi identities.  More concretely, imposing the vanishing of  $Jac(X_2, JX_2, Y)$ for $Y = X_4, JX_4$, $Jac(X_3, JX_2, Y)$ for $Y = X_4, JX_4$, and $Jac(X_4, JX_4, JX_1)$, we get $JX_1\in\frg_4$, which is not possible.

\smallskip\noindent
Due to the previous discussion, the case dim\,$\frg_4=5$ is not valid.

\smallskip\noindent
$\sqbullet$ \ Assume dim\,$\frg_4=6$. Repeating a similar argument, one can show that this case can neither hold. Indeed, it suffices to consider the cases $J\frg_2\subset \frg_4$ and its opposite, and impose the vanishing of the Jacobi identities in order to obtain contradictions with the assumed arrangement of the ascending central series.
\end{addmargin}
\vspace{-0.2cm}
\end{proof}

As a consequence of this result, we obtain the possible structure of the term $\frg_2$:

\begin{corollary}\label{posibles_g2}
Let $(\frg,J)$ be an $8$-dimensional NLA admitting a strongly non-nilpotent complex structure. Then,
$$\text{dim } \frg_{1}=1, \ \text{and}\ \
    \frg_{2}=\begin{cases}
    \frg_1\oplus \mathcal V_2, \\
    \frg_1\oplus \mathcal V_2\oplus \langle Y \rangle,\\
    \frg_1\oplus \mathcal V_2\oplus \mathcal V_3,
    \end{cases}$$
where $Y\notin J\frg_1$ and $\mathcal V_2,\,\mathcal V_3$ are $J$-invariant subspaces of dimension $2$ (see Notation~\ref{notation}).
\end{corollary}

\subsection{The terms $\frg_3$ and $\frg_4$}\label{sec:g34}
\noindent
Once we have obtained the structure of $\frg_2$, we next describe the spaces~$\frg_3$ and~$\frg_4$ depending on the dimension of~$\frg_2$.  As a final result, we obtain a structural theorem for the ascending central series (see Theorem~\ref{teorema-estructura-acs}).

Let us remark that every case of Corollary \ref{posibles_g2} implies $\frg_1\oplus \mathcal V_2\subseteq\frg_2$.
Hence, we can set
$$\frg_1=\langle X_1\rangle,\quad \mathcal V_2 = \langle X_2, JX_2\rangle.$$
This determines a family $\{ X_i, JX_i \}_{i=1}^2$ that will be completed up to
a $J$-adapted basis $\mathcal B$ of the Lie algebra~${\mathfrak g}$.
In addition, the Lie brackets within the previous family satisfy:
$$[X_1,\cdot\ ]=0,\quad [X_2,\cdot\ ],\, [JX_{2},\cdot\ ] \in \langle X_1 \rangle.$$

\subsubsection{Study for dim\,$\frg_2 = 3$}

In this case, one clearly has $\frg_1\oplus \mathcal V_2=\frg_2$.  As a first observation for $\frg_3$, we prove the following result:

\begin{lemma}\label{lema1}
Let $(\frg,J)$ be an $8$-dimensional NLA endowed with an SnN complex structure. If
$\dim \frg_2=3$, then $\frg_3\neq \frg_{2}\oplus J\frg_{1}$.
\end{lemma}

\begin{proof}
We argue by contradiction.  Let us suppose that  $\frg_3= \frg_{2}\oplus J\frg_{1}=\mathcal V_1\oplus \mathcal V_2$.  Then, we can assume $\mathcal V_3\cap \frg_4\neq \{0\}$, i.e., there exists an element $X_3$ in $\mathcal B$ belonging to $\frg_4$.  The ascending central series starts as follows:
$$\frg_1=\langle X_1\rangle,\quad \frg_2=\langle X_1,X_2,JX_{2}\rangle, \quad \frg_3=\langle X_1,X_{2},JX_{1},JX_{2}\rangle, \quad \frg_4\supseteq\langle X_1,X_{2}, X_3,JX_{1},JX_{2}\rangle.$$  Up to this moment, we have fixed six elements in $\mathcal B$, namely, $\{X_i, JX_i\}_{i=1}^3$.  Moreover, $\frg\setminus (\mathcal V_1\oplus \mathcal V_2\oplus \mathcal V_3)$ is a $J$-invariant space of dimension 2 that we will call $\mathcal V_4$. We will consider two different situations: $\mathcal V_4\cap \frg_4\neq \{0\}$ and $\mathcal V_4\cap \frg_4= \{0\}$.

\smallskip\noindent
$\bullet$ \ \textsc{Study of} $\mathcal V_4\cap \frg_4\neq \{0\}$.

\smallskip\noindent
Due to the previous hypothesis, there exists a new element in $\frg_4$ linearly independent with $\{X_i, JX_i\}_{i=1}^3$
that we will denote $X_4$. In this way,
$$\frg_4\supseteq\langle X_1,X_{2}, X_3, X_4, JX_{1},JX_{2}\rangle.$$
By Lemma~\ref{lemma1_serie_asc}, part \textrm{(i)} for $k=4$, we have $\frg_4=\frg$.  Hence, it is possible to express the Lie brackets as follows, making use of the Nihenjuis condition and the nilpotency of $\frg$:
\begin{equation}\label{bracket_JX2}
\begin{array}{ll}
[X_1,\cdot\  ]=0,&[X_{3},JX_{k}]_{k=3,4} = \sum_{i=1}^2 (b_{3k}^{i}\,X_{i}+c_{3k}^{i}\,JX_{i}),\\[3pt]
[X_2,X_k]_{k=3,4} =a_{2k}^1\,X_1, & [X_4, JX_1] = b_{41}^{2}\,X_{2}+c_{41}^{2}\,JX_{2},\\[3pt]
[X_2,JX_1]=0, & [X_{4},JX_{2}]=b_{24}^{1}\,X_{1},\\[3pt]
[X_2,JX_k]_{k=2,3,4} =b_{2k}^1\,X_1,& [X_{4},JX_{k}]_{k=3,4} = \sum_{i=1}^2 (b_{4k}^{i}\,X_{i}+c_{4k}^{i}\,JX_{i}),\\[3pt]
[X_3, X_4]=\sum_{i=1}^2 (a_{34}^{i}\,X_{i}+\alpha_{34}^{i}\,JX_{i}),& [JX_1,JX_2]=0, \\[3pt]
[X_3, JX_1] = b_{31}^{2}\,X_{2}+c_{31}^{2}\,JX_{2},&[JX_1,JX_k]_{k=3,4} = c_{k1}^{2}\,X_{2}-b_{k1}^{2}\,JX_{2},\\[3pt]
[X_{3},JX_{2}]=b_{23}^{1}\,X_{1},&[JX_{2},JX_{k}]_{k=3,4}=a_{2k}^{1}\,X_{1},\\[3pt]
\multicolumn{2}{l}{
[JX_3,JX_4] = \sum_{i=1}^2(a_{34}^{i}-c_{34}^{i}+c_{43}^{i})\,X_{i}+\sum_{i=1}^2(\alpha_{34}^{i}+b_{34}^{i}-b_{43}^{i})\,JX_{i}.
}
\end{array}
\end{equation}

The idea is now to impose the vanishing of the Jacobi identity \eqref{Jacobi} in order to get a contradiction.
We sketch how to proceed.  Consider $Jac(X_3, JX_3, Y)$ for $Y= X_4, JX_4$ and  $Jac(X_4, JX_4, Y)$ for $Y= X_3, JX_3$ to obtain the following homogeneous systems of equations:
$$A
\begin{pmatrix} b_{31}^2 \\ c_{31}^2 \end{pmatrix} =
\begin{pmatrix} 0 \\ 0 \end{pmatrix},\quad A
\begin{pmatrix} b_{41}^2 \\ c_{41}^2 \end{pmatrix} =
\begin{pmatrix} 0 \\ 0 \end{pmatrix},\quad \text{where}\quad A=\begin{pmatrix}c_{43}^1-c_{34}^1& b_{34}^1-b_{43}^1\\ -(b_{34}^1-b_{43}^1)&c_{43}^1-c_{34}^1\end{pmatrix}.$$
In order to preserve the arrangement of the ascending central series, the determinant of $A$ must be zero, i.e., $c_{43}^1 = c_{34}^1$ and $b_{43}^1 = b_{34}^1$.  With these values of the parameters, we go back to the previous four Jacobi identities and obtain:
$$B
\begin{pmatrix} b_{31}^2 \\ c_{31}^2 \end{pmatrix} = c_{33}^1
\begin{pmatrix} c_{41}^2 \\ b_{41}^2 \end{pmatrix},
\quad
B
\begin{pmatrix} c_{41}^2 \\ b_{41}^2 \end{pmatrix} = c_{44}^1
\begin{pmatrix} b_{31}^2 \\ c_{31}^2 \end{pmatrix},
\quad \text{where }
B=\begin{pmatrix}
\alpha_{34}^1 & c_{34}^1 \\ c_{34}^1 & -\alpha_{34}^1
\end{pmatrix}.$$
The contradiction appears here, since it is not possible to have determinant of $B$ either equal to zero or different from zero if we want to preserve the desired ascending central series.

\smallskip\noindent
$\bullet$ \ \textsc{Study of} $\mathcal V_4\cap \frg_4=\{0\}$.

\smallskip\noindent
First, let us note that $\frg_5=\frg$.
The proof will be finished after analyzing two cases: $\mathcal V_3\subset \frg_4$ and $\mathcal V_3\not\subset \frg_4$.

\smallskip\noindent
\begin{addmargin}[1.5em]{0em}
$\sqbullet$ \
Let us consider $\mathcal V_3\subset \frg_4$.  This is equivalent to $\frg_4=\frg_3\oplus \mathcal V_3$, so in particular
we observe that $\frg_4$ is $J$-invariant.
The Lie brackets are the same as in \eqref{bracket_JX2} with the only exception of
$$[X_4, JX_4] = b_{44}^1 X_1 + b_{44}^2 X_2 + b_{44}^3 X_3 + c_{44}^1 JX_1 + c_{44}^2 JX_2 + c_{44}^3  JX_3.$$
Let us remark that one needs $(b_{44}^3, c_{44}^3)\neq(0,0)$ in order to ensure $\frg_4\neq \frg$.  We again compute different
Jacobi identities to get a contradiction.  More concretely, focusing on the triplets
$(X_2, JX_1, Y)$, where $Y=X_3, X_4, JX_3, JX_4$,
we can conclude that $b_{22}^1=0$.  Furthermore, from
$(X_4, JX_4, Y)$ with $Y=X_2,JX_1, JX_2$,
 we obtain $a_{23}^1=b_{23}^1=b_{31}^2=c_{31}^2=0$.  Now, look at
$Jac(X_3, JX_3,Y)$, for $Y=X_4, JX_4$,
 to get $b_{33}^2=c_{33}^2=0$.  Finally, working again with
 $Jac(X_4, JX_4,Y)$, this time for $Y=X_3, JX_3$,
 we reach the desired contradiction.

\smallskip\noindent
$\sqbullet$ \ We now focus on the case $\mathcal V_3\not\subset \frg_4$, i.e., $\frg_4=\frg_3\oplus \langle X_3\rangle.$
We remark that the Lie brackets agree with~\eqref{bracket_JX2}, except for $[X_4, JX_4]$ (which coincides with that defined a
few lines above),  and $[X_4, JX_3]$, $[JX_3, JX_4]$ (which belong to $\frg_4$).  More precisely, let
$$[X_{4},JX_{3}]=b_{43}^1\,X_{1}+b_{43}^2\,X_{2}+b_{43}^3\,X_{3}+c_{43}^1\,JX_{1}+c_{43}^2\,JX_{2},$$
where $b_{43}^i, b_{43}^3, c_{43}^i\in\mathbb{R}$, for $i=1,2$. Applying the Nijenhuis condition,
one obtains
$$
[JX_{3},JX_{4}] = \sum_{i=1}^2(a_{34}^{i}-c_{34}^{i}+c_{43}^{i})\,X_{i}+\sum_{i=1}^2(\alpha_{34}^{i}+b_{34}^{i}-b_{43}^{i})\,JX_{i}-b_{43}^3\,JX_{3}.
$$
Since our Lie algebra $\frg$ is nilpotent, we necessarily have $b_{43}^{3}=0$. Nevertheless,
this choice makes $JX_{3}\in\frg_{4}$, which is not allowed.
\end{addmargin}
\smallskip
This completes the proof of the lemma.
\end{proof}


\begin{lemma}\label{lema2}
Let $(\frg,J)$ be an $8$-dimensional NLA endowed with an SnN complex structure.
If $\dim \frg_2=3$ and $\frg_3\cap J\frg_{1}\neq \{0\}$,
then the ascending central series of $\frg$ satisfies:
$$\text{dim }\frg_1 = 1,\quad \frg_2 = \frg_1\oplus\mathcal V_2,\quad \frg_3\text{ is $J$-invariant with dim } \frg_3= 6 \text{ or }8,$$ being $\mathcal V_2$ a $2$-dimensional $J$-invariant space.
\end{lemma}

\begin{proof}
Making use of the hypothesis and Lemma~\ref{lema1}, we can find $X_1, X_2, X_3\in\frg$ such that:
$$\frg_{1}=\langle X_{1}\rangle,\quad \frg_{2}=\langle X_{1},X_{2},JX_{2}\rangle,\quad \frg_{3}\supseteq\langle X_{1},X_{2},X_{3},JX_{1},JX_{2}\rangle.$$
We consider $\{X_i, JX_i\}_{i=1}^3$ belonging to the $J$-adapted basis $\mathcal B$ and $X_4, JX_4$ generators of the $J$-invariant subspace $\mathcal V_4 =\frg\setminus (\mathcal V_1\oplus \mathcal V_2\oplus \mathcal V_3)$.
We can set the following brackets, where we have already applied the Nijenhuis condition and the nilpotency of $\frg$:
\begin{equation}\label{bracket_lema2}
\begin{array}{ll}
[X_1,\cdot\  ]=0,& [X_4, JX_1] = b_{41}^{2}\,X_{2}+c_{41}^{2}\,JX_{2},\\[2pt]
[X_2,X_k]_{k=3,4}=a_{2k}^1\,X_1,& [X_{4},JX_{2}]=b_{24}^{1}\,X_{1},\\[2pt]
[X_2,JX_1]=0, & [X_{4},JX_{k}]_{k=3,4} \text{ unknown},\\[2pt]
[X_2,JX_k]_{k=2,3,4}=b_{2k}^1\,X_1, &[JX_1,JX_2]=0, \\[2pt]
[X_3, X_4]=a_{34}^{1}\,X_{1}+a_{34}^{2}\,X_{2}+\alpha_{34}^{2}\,JX_{2},&[JX_1,JX_k]_{k=3,4}  = c_{k1}^{2}\,X_{2}-b_{k1}^{2}\,JX_{2}, \\[2pt]
[X_3, JX_1] =b_{31}^{2}\,X_{2}+c_{31}^{2}\,JX_{2},&[JX_{2},JX_{k}]_{k=3,4} =a_{2k}^{1}\,X_{1}, \\[2pt]
[X_{3},JX_{2}]=b_{23}^{1}\,X_{1},&[JX_{3},JX_{4}] \text{ unknown},\\[2pt]
[X_{3},JX_{k}]_{k=3,4} = b_{3k}^1 X_1 + b_{3k}^2 X_2  + c_{3k}^2 JX_2,&
\end{array}
\end{equation}
where all the coefficients are real numbers.

We should accomplish the study of~$\frg_3$.  The idea is the following one:  first, we will suppose that~$\frg_3$ is not $J$-invariant and contradict the assumption;  afterwards, we will show that valid solutions exist when~$\frg_3$ is $J$-invariant of dimension both 6 and~8.

\smallskip\noindent
$\bullet$ \ \textsc{Study when} $\frg_3$ \textsc{is not} $J$-\textsc{invariant}.

\smallskip\noindent
The possible dimensions for this space are 5 and 6. Let us separately analyze each case.

\smallskip\noindent
\begin{addmargin}[1.5em]{0em}
$\sqbullet$ \ Let dim\,$\frg_3=6$.  This implies that $\mathcal V_4\cap \frg_3\neq \{0\}$, and we can set: $$\frg_{3}=\langle X_{1},X_{2},X_{3},X_4, JX_{1},JX_{2}\rangle,\quad \frg_4 = \frg.$$
Hence, the unknown brackets in~\eqref{bracket_lema2} can be determined:
$$\begin{array}{l}
[X_{4},JX_{k}]_{k=3,4}=b_{4k}^{1}\,X_{1}+b_{4k}^{2}\,X_{2}+c_{4k}^{2}\,JX_{2}, \\[2pt]
[JX_{3},JX_{4}]=a_{34}^{1}\,X_{1}+(a_{34}^{2}-c_{34}^{2}+c_{43}^{2})\,X_{2}+(b_{34}^{1}-b_{43}^{1})\,JX_{1}+
(\alpha_{34}^{2}+b_{34}^{2}-b_{43}^{2})\,JX_{2}.
\end{array}$$
Observe that  $b_{43}^{1}\neq b_{34}^{1}$, or otherwise $\frg_3 = \frg$.  Computing the Jacobi identities
$Jac\,(JX_{3},JX_{4},Y)$ with $Y=X_3, X_4$,
we are led to consider $b_{31}^{2}=b_{41}^{2}=c_{31}^{2}=c_{41}^{2}=0$, but then one observes that indeed $JX_{1}\in\frg_{1}$, which is a contradiction.

\smallskip\noindent
$\sqbullet$ \ We now assume dim\,$\frg_3=5$. As $J\frg_3\neq\frg_3$, we have
$$\frg_{3}=\langle X_{1},X_{2},X_{3},JX_{1},JX_{2}\rangle,$$
and we need to turn our attention to $\frg_4$. Two different situations arise: $\mathcal V_4\cap \frg_4\neq\{0\}$
and $\mathcal V_4\cap \frg_4=\{0\}$.

\smallskip
\begin{addmargin}[1.5em]{0em}
$\triangleright$ \
If $\mathcal V_4\cap \frg_4\neq\{0\}$, the unknown real Lie brackets can be set as follows:

\medskip
\begin{center}
$\begin{array}{l}
[X_{4},JX_{3}]=b_{43}^{1}\,X_{1}+b_{43}^{2}\,X_{2}+c_{43}^{1}\,JX_{1}+c_{43}^{2}\,JX_{2},\\[2pt]
[X_{4},JX_{4}]=b_{44}^{1}\,X_{1}+b_{44}^{2}\,X_{2}+b_{44}^{3}\,X_{3}+c_{44}^{1}\,JX_{1}+c_{44}^{2}\,JX_{2},\\[2pt]
[JX_{3},JX_{4}]=(a_{34}^{1}+c_{43}^{1})\,X_{1}+(a_{34}^{2}-c_{34}^{2}+c_{43}^{2})\,X_{2}+
    (b_{34}^{1}-b_{43}^{1})\,JX_{1}  +(\alpha_{34}^{2}+b_{34}^{2}-b_{43}^{2})\,JX_{2}.\end{array}$
\end{center}

\medskip\noindent
Observe that one necessarily has $\frg_4=\frg$.
The contradiction arises from the Jacobi identity.  More concretely, imposing the vanishing of
$Jac\,(X_{2},JX_{1},Y)$, for $Y= X_3, JX_3, X_4, JX_4$, one obtains $b_{22}^{1}=0$, in order to ensure that $JX_1\notin \frg_1$. Next, if we compute $Jac\,(X_{3},JX_{3},Y)$ for $Y=X_4, JX_4$ and $Jac\,(X_{4},JX_{4},Z)$ for $Z=X_3, JX_1$, we must assume $b_{31}^{2}=c_{31}^{2}=0$ if we want to preserve the arrangement of the series.
In a similar way, $Jac\,(X_{3},JX_{1},Y)$ with $Y=X_4, JX_4$
imply $a_{23}^{1}=b_{23}^{1}=0$. Coming back to $Jac\,(X_{3},JX_{3},X_{4})$ and $Jac\,(X_{3},JX_{3},JX_{4})$,
one gets $b_{33}^{2}=c_{33}^{2}=0$.  Finally, taking into account the previous fixed values, the contradiction comes from
$Jac(X_4, JX_4, JX_3)$.

\smallskip\noindent
$\triangleright$ \
If $\mathcal V_4\cap \frg_4=\{0\}$, the structure of the ascending central series is:
$$\frg_1=\langle X_1\rangle,\quad \frg_2=\langle X_1,X_2,JX_{2}\rangle,\quad
 \frg_3=\langle X_1,X_{2},X_{3},JX_{1},JX_{2}\rangle,$$
$$\frg_4=\langle X_1,X_{2},X_{3},JX_{1},JX_{2},JX_{3}\rangle, \quad \frg_{5}=\frg.$$
We note that the only different bracket between this case and the previous one is
$$[X_{4},JX_{4}]=b_{44}^{1}\,X_{1}+b_{44}^{2}\,X_{2}+b_{44}^{3}\,X_{3}
    +c_{44}^{1}\,JX_{1}+c_{44}^{2}\,JX_{2}+c_{44}^{3}\,JX_{3},$$
where $b_{44}^{i},c_{44}^{i}\in\mathbb{R}$, for $i=1,2,3$. For this reason, one needs $c_{44}^{3}\neq 0$
in order to effectively separate the two situations.
Repeating the same Jacobi computations as above, we conclude that this case is neither valid.
\end{addmargin}
\end{addmargin}
\noindent
This proves that ${\mathfrak g}_3$ must be $J$-invariant.

\smallskip\noindent
$\bullet$ \ \textsc{Study when} $\frg_3$ \textsc{is} $J$-\textsc{invariant}.

\smallskip\noindent
There are two possibilities for $\frg_{3}$: either
dim\,$\frg_{3}=6$ or dim\,$\frg_{3}=8$. In what follows, we will see that in both cases one can find an appropriate choice of the parameters that fulfills the Jacobi identity and preserves the desired ascending central series.

\smallskip\noindent
\begin{addmargin}[1.5em]{0em}
$\sqbullet$ \ If dim\,$\frg_3=6$, then
$$\frg_3=\langle X_1, X_2, X_3, JX_1, JX_2, JX_3\rangle,\quad \frg_4 = \frg.$$
The unknown brackets in~\eqref{bracket_lema2} can be expressed as:
\begin{equation*}
\begin{array}{l}
[X_{4},JX_{3}]=b_{34}^{1}\,X_{1}+b_{43}^{2}\,X_{2}+c_{43}^{2}\,JX_{2},\\[2pt]
[X_{4},JX_{4}]=b_{44}^{1}\,X_{1}+b_{44}^{2}\,X_{2}+b_{44}^{3}\,X_{3}+c_{44}^{1}\,JX_{1}+
   c_{44}^{2}\,JX_{2}+c_{44}^{3}\,JX_{3}, \\[2pt]
[JX_{3},JX_{4}]=a_{34}^{1}\,X_{1}+(a_{34}^{2}-c_{34}^{2}+c_{43}^{2})\,X_{2}+
  (\alpha_{34}^{2}+b_{34}^{2}-b_{43}^{2})\,JX_{2}.
\end{array}
\end{equation*}
From $Jac(X_2,JX_1,Y)$ for $Y=X_3, X_4, JX_3, JX_4$, we are able to conclude that $b^1_{22}=0$, since
$(b^2_{31},c^2_{31},b^2_{41},c^2_{41})\neq (0,0,0,0)$ to ensure $JX_1\notin\frg_1$. Then, calculating
$Jac(X_4, JX_4, Y)$ with $Y=X_2,JX_1,JX_2$, two systems of equations arise:
$$C
\begin{pmatrix} a^1_{23} \\ b^1_{23} \end{pmatrix}
=
\begin{pmatrix} 0 \\ 0 \end{pmatrix},
\quad
C
\begin{pmatrix} c^2_{31} \\ b^2_{31} \end{pmatrix}
=
\begin{pmatrix} 0 \\ 0 \end{pmatrix},
\text{ where }
C=\begin{pmatrix} b^3_{44} & c^3_{44} \\ c^3_{44} & - b^3_{44} \end{pmatrix}.
$$
Using $Jac(X_4,JX_4, Y)$ for $Y=X_3,JX_3$ when $(b^3_{44},c^3_{44})\neq(0,0)$, and
$Jac(X_4,JX_4,X_3)$, $Jac(X_3,JX_1,Y)$, $Jac(X_3,JX_3,Y)$, for $Y=X_4,JX_4$, when $(b^3_{44},c^3_{44})=(0,0)$,
one obtains
$$a_{23}^{1}=b_{23}^{1}=b_{31}^{2}=b_{33}^{2}=c_{31}^{2}=c_{33}^{2}=0.$$
We note that 19 structure constants are still undetermined. They must be chosen in such a way that the desired
ascending central series is preserved and the 3 equations given by the non-trivial Jacobi identities are fulfilled.
Here, we simply note that particular solutions satisfying the two previous
requirements can be found.

\smallskip\noindent
$\sqbullet$ \
If dim\,$\frg_3=8$, then $\frg_3=\frg$ and the unknown brackets in~\eqref{bracket_lema2} are given by:
\begin{equation*}
\begin{array}{l}
[X_{4},JX_{3}]=b_{34}^{1}\,X_{1}+b_{43}^{2}\,X_{2}+c_{43}^{2}\,JX_{2}, \\[2pt]
[X_{4},JX_{4}]=b_{44}^{1}\,X_{1}+b_{44}^{2}\,X_{2}+c_{44}^{2}\,JX_{2}, \\[2pt]
[JX_{3},JX_{4}]=a_{34}^{1}\,X_{1}+(a_{34}^{2}-c_{34}^{2}+c_{43}^{2})\,X_{2}+(\alpha_{34}^{2}+b_{34}^{2}-b_{43}^{2})\,JX_{2},
\end{array}
\end{equation*}
where $b_{44}^{1}, b_{4k}^{2}, c_{4k}^{2}\in\mathbb{R}$, for $k=1,3,4$.  Let us study the Jacobi identity \eqref{Jacobi}.  From
$Jac(X_2,JX_1,Y)$, for $Y=X_3,X_4,JX_3,JX_4$, we get:
$$b_{22}^{1}\,c_{31}^{2} = b_{22}^{1}\,c_{41}^{2} = b_{22}^{1}\,b_{31}^{2} = b_{22}^{1}\,b_{41}^{2} = 0.$$
Since $(b_{31}^{2},c_{31}^{2},b_{41}^{2},c_{41}^{2})\neq(0,0,0,0)$
is needed to ensure $JX_1\notin\frg_1$, one immediately concludes $b_{22}^{1}=0$.
The remaining Jacobi identities give
a quadratic system with 8 equations and 22 unknowns
that admits particular solutions fulfilling all the prerequisites.
\end{addmargin}

\smallskip\noindent
This concludes the proof of the lemma.
\end{proof}


\begin{lemma}\label{lema3}
Let $(\frg,J)$ be an $8$-dimensional NLA endowed with an SnN complex structure. If
$\dim \frg_2=3$ and $\frg_3\cap J\frg_1= \{0\}$, then the ascending central series of $\frg$ satisfies:
$$\text{dim }\frg_1 = 1,\quad \frg_2 = \frg_1\oplus\mathcal V_2, \quad \frg_3 = \frg_1\oplus\mathcal V_2\oplus\mathcal V_3, \quad \frg_4\text{ is $J$-invariant},$$ where $\mathcal V_2, \mathcal V_3$ are $2$-dimensional $J$-invariant spaces.
\end{lemma}

\begin{proof}
By hypothesis, one can find $X_1, X_2, X_3\in\frg$ such that:
$$\frg_{1}=\langle X_{1}\rangle,\quad \frg_{2}=\langle X_{1},X_{2},JX_{2}\rangle,\quad \frg_{3}\supseteq\langle X_{1},X_{2},X_{3},JX_{2}\rangle.$$
Furthermore, we know that $JX_1\notin \frg_3$, so $\frg$ must be at least 4-step nilpotent.
Let us take $\{X_i, JX_i\}_{i=1}^3$ in the $J$-adapted basis $\mathcal B$ and
$X_4, JX_4$ generators of the $J$-invariant subspace
$\mathcal V_4 =\frg\setminus (\mathcal V_1\oplus \mathcal V_2\oplus \mathcal V_3)$.
The following real brackets can be set, bearing in mind the Nijenhuis condition and the nilpotency of $\frg$:
\begin{equation*}
\begin{array}{ll}
[X_1,\cdot\  ]=0,& [X_4, JX_1] \text{ unknown},\\[2pt]
[X_2,X_k]_{k=3,4}=a_{2k}^1\,X_1, & [X_{4},JX_{2}]=b_{24}^{1}\,X_{1},\\[2pt]
[X_2,JX_1]=0, & [X_{4},JX_{k}]_{k=3,4} \text{ unknown},\\[2pt]
[X_2,JX_k]_{k=2,3,4}=b_{2k}^1\,X_1, &[JX_1,JX_2]=0, \\[2pt]
[X_3, X_4]=a_{34}^{1}\,X_{1}+a_{34}^{2}\,X_{2}+\alpha_{34}^{2}\,JX_{2},&[JX_1,JX_3] = c_{31}^{2}\,X_{2}-b_{31}^{2}\,JX_{2}, \\[2pt]
[X_3, JX_1] =b_{31}^{2}\,X_{2}+c_{31}^{2}\,JX_{2},&[JX_1,JX_4] \text{ unknown}, \\[2pt]
[X_{3},JX_{2}]=b_{23}^{1}\,X_{1},&[JX_{2},JX_{k}]_{k=3,4} =a_{2k}^{1}\,X_{1}\\[2pt]
[X_{3},JX_{k}]_{k=3,4} = b_{3k}^1 X_1 + b_{3k}^2 X_2  + c_{3k}^2 JX_2,&[JX_{3},JX_{4}] \text{ unknown}.\\[2pt]    \end{array}
\end{equation*}

To prove the result it suffices to follow the next scheme. First, one checks by contradiction that $\mathcal V_4\cap~\frg_3=~\{0\}$.
As a consequence, one has $4\leq \text{\,dim\,}\frg_3\leq 5$.  Then, the 4-dimensional case is rejected
(studying $\mathcal V_4\cap \frg_4\neq\{0\}$ and $\mathcal V_4\cap \frg_4=\{0\}$), and finally one shows
that solutions exist for the 5-dimensional case.
This completes the study of every possibility, and thus our result is attained.
\end{proof}

As a consequence of the previous three Lemmas \ref{lema1}, \ref{lema2}, and \ref{lema3},
we reach the following structure result.


\begin{proposition}\label{serie-dimg23}
Let $(\frg,J)$ be an $8$-dimensional NLA endowed with an SnN complex structure. If
$\dim \frg_2=3$, then the ascending central series of $\frg$ satisfies:
$$\text{dim }\frg_1 = 1,\quad \frg_2 = \frg_1\oplus\mathcal V_2,$$ and one of the following possibilities:
\begin{itemize}
\item[(i)]  $\frg_{3}=\frg$;
\item[(ii)] $\frg_{3}=\mathcal V_1\oplus \mathcal V_2\oplus \mathcal V_3$, \ $\frg_{4}=\frg$;
\item[(iii)] $\frg_{3}=\frg_1\oplus \mathcal V_2\oplus \mathcal V_3$, \ $\frg_{4}=\frg$;
\item[(iv)] $\frg_{3}=\frg_1\oplus \mathcal V_2\oplus \mathcal V_3$, \
              $\frg_{4}=\mathcal V_1\oplus \mathcal V_2\oplus \mathcal V_3$, \ $\frg_{5}=\frg$,
\end{itemize}
where $\mathcal V_i$ are $2$-dimensional $J$-invariant spaces.
\end{proposition}

\subsubsection{Study for dim\,$\frg_2 = 4$}

\noindent
After the analysis of dim\,$\frg_2 = 3$, let us suppose that dim~$\frg_2=4$.
As a consequence of Corollary \ref{posibles_g2}, we can find $X_1, X_2, X_3\in\frg$ such that:
$$\frg_{1}=\langle X_{1}\rangle,\quad  \frg_{2}=\langle X_{1},X_{2},X_{3},JX_{2}\rangle.$$
We consider $\{X_i, JX_i\}_{i=1}^3$ in the $J$-adapted basis $\mathcal B$ and denote $X_4, JX_4$ the
generators of the $J$-invariant subspace $\mathcal V_4 =\frg\setminus (\mathcal V_1\oplus \mathcal V_2\oplus \mathcal V_3)$.
The Lie brackets of $\frg$ have the following form:
\begin{equation}\label{bracket_dim45}
\begin{array}{lll}
[X_1,\cdot\  ]=0,&[X_3, JX_1] =0,& [X_{4},JX_{k}]_{k=3,4} \text{ unknown},\\[2pt]
[X_2,X_k]_{k=3,4}=a_{2k}^1\,X_1,& [X_{3},JX_{2}]=b_{23}^{1}\,X_{1},&[JX_1,JX_k]_{k=2,3}=0,\\[2pt]
[X_2,JX_1]=0, & [X_{3},JX_{k}]_{k=3,4} = b_{3k}^1 X_1, &[JX_1,JX_4] \text{ unknown},\\[2pt]
[X_2,JX_k]_{k=2,3,4}=b_{2k}^1\,X_1, &[X_4, JX_1] \text{ unknown},&[JX_{2},JX_{k}]_{k=3,4} =a_{2k}^{1}\,X_{1},\\[2pt]
[X_3, X_4]=a_{34}^{1}\,X_{1},& [X_{4},JX_{2}]=b_{24}^{1}\,X_{1},&[JX_{3},JX_{4}] \text{ unknown},   \end{array}
\end{equation}
where all the coefficients are real numbers.

\begin{proposition}\label{serie-dimg24}
Let $(\frg,J)$ be an $8$-dimensional NLA endowed with an SnN complex structure. If
$\dim \frg_2=4$, then the ascending central series of $\frg$ satisfies:
$$\text{dim }\frg_1 = 1,\quad \frg_2 = \frg_1\oplus\mathcal V_2\oplus\langle Y \rangle,\quad \frg_3 \text{ is } J\text{-invariant with dim } \frg_3= 6 \text{ or }8,$$ where $Y\notin J\frg_1$ and $\mathcal V_2$ is a $2$-dimensional $J$-invariant space.
\end{proposition}

\begin{proof}
As a consequence of the previous lines we directly focus on $\frg_3$.  We will first see by contradiction that $\frg_3$ must be $J$-invariant and then provide two valid arrangements of the ascending central series.

\smallskip\noindent
$\bullet$ \ \textsc{Study when} $\frg_3$ \textsc{is not} $J$\textsc{-invariant}.

\smallskip\noindent
Due to the nilpotency of $\frg$, one has $5\leq \text{dim } \frg_3\leq 6$.  We next analyze these two possibilities.

\smallskip\noindent
\begin{addmargin}[1.5em]{0em}
$\sqbullet$ \ Let dim\,$\frg_3=6$. As a consequence of our assumptions, one
directly has $\mathcal V_4\cap \frg_3\neq \{0\}$, i.e.,
$$\frg_1=\langle X_1\rangle, \qquad \frg_2=\langle X_1,X_2,X_3,JX_2\rangle, \qquad
  \frg_3\supseteq\langle X_1,X_2,X_3,X_4,JX_2\rangle.$$
 This arrangement together with the Nijenhuis condition allows us to fix the unknown real brackets in~\eqref{bracket_dim45}:
\begin{equation}\label{bracket_caso1121}
\begin{array}{l}
[X_4,JX_1]=b_{41}^2\,X_2+b_{41}^3\,X_3+c_{41}^2\,JX_2, \\[2pt]
[X_4,JX_3]=b_{43}^1\,X_1+b_{43}^2\,X_2+c_{43}^2\,JX_2, \\[2pt]
[X_4,JX_4]=b_{44}^1\,X_1+b_{44}^2\,X_2+b_{44}^3\,X_3+c_{44}^2\,JX_2, \\[2pt]
[JX_1,JX_4]=c_{41}^2\,X_2-b_{41}^2\,JX_2-b_{41}^3\,JX_3, \\[2pt]
[JX_3,JX_4]=a_{34}^1\,X_1+c_{43}^2\,X_2+(b_{34}^1-b_{43}^1)\,JX_1-b_{43}^2\,JX_2.
\end{array}
\end{equation}
Moreover, $\frg_3$ must contain a 4-dimensional $J$-invariant subspace.  After an arrangement of generators, we can ensure that one of the spaces $\mathcal V_i$, where $i=1,3,4$, lies in $\frg_3$.  We separately study each possibility, discarding all of them:

\smallskip\noindent
\begin{addmargin}[1.5em]{0em}
$\triangleright$ \ If $\mathcal V_4\subset\frg_3$, then we see from~\eqref{bracket_caso1121} that $\frg_3 = \frg$, which contradicts our assumption.

\smallskip\noindent
$\triangleright$ \ Let $\mathcal V_3\subset\frg_3$. The Lie brackets in~\eqref{bracket_caso1121} tell us that $b_{43}^1=b_{34}^1$
and $b_{41}^3\neq 0$ (to ensure $\frg_3\neq\frg$).  Moreover, from $Jac(X_4, JX_4, JX_1)$ we need $b_{43}^2 = c_{43}^2 = 0$,
but this choice makes $JX_3\in\frg_2$, which is a contradiction.

\smallskip\noindent
$\triangleright$ \ Suppose that $\mathcal V_1\subset\frg_3$. 
Then, from~\eqref{bracket_caso1121} we find $b^3_{41}=0$ and $b^1_{43} \neq b^1_{34}$.
Using $Jac(X_4, JX_4, JX_3)$ we also get $b_{41}^2=c_{41}^2=0$.  However,
these values lead to $JX_1\in\frg_1$, which is not possible.
\end{addmargin}

\smallskip\noindent
$\sqbullet$ \ Let dim\,$\frg_3=5$. Notice that $\mathcal V_4\cap \frg_3 = \{0\}$ (otherwise, one would apply Lemma~\ref{lemma1_serie_asc} \textrm{(ii)} with~$r=3$ and reach a contradiction), so $\frg_3$ must contain a 4-dimensional $J$-invariant subspace.  Similarly to the previous case, we can assume that one of the spaces $\mathcal V_1$ or $\mathcal V_3$ is contained in $\frg_3$.  Again, none of the situations is possible.

\smallskip\noindent
\begin{addmargin}[1.5em]{0em}
$\triangleright$ \ Let us suppose that $\mathcal V_1\subset\frg_3$.  The first terms of the ascending central series are
$$\frg_1=\langle X_1\rangle, \qquad \frg_2=\langle X_1,X_2,X_3,JX_2\rangle, \qquad
  \frg_3=\langle X_1,X_2,X_3, JX_1, JX_2\rangle.$$
 The unknown brackets in~\eqref{bracket_dim45} can be set as follows. The integrability of~$J$ and the
 nilpotency of~$\frg$ imply that
  $$\begin{array}{l}
 [X_4,JX_1]=b_{41}^2\,X_2+c_{41}^2\,JX_2,\\[2pt]
 [JX_1,JX_4]=c_{41}^2\,X_2-b_{41}^2\,JX_2.
 \end{array}$$
 In addition, if either $\mathcal V_4\cap \frg_4\neq\{0\}$ or $\mathcal V_4\cap \frg_4=\{0\}$, it is easy to see that
 $$\begin{array}{l}
 [X_4,JX_3]=b_{43}^1\,X_1+b_{43}^2\,X_2+c_{43}^1\,JX_1+c_{43}^2\,JX_2,\\[2pt]
 [X_4,JX_4]=b_{44}^1\,X_1+b_{44}^2\,X_2+b_{44}^3\,X_3+c_{44}^1\,JX_1+c_{44}^2\,JX_2 +c_{44}^3\,JX_3, \\[2pt]
 [JX_3,JX_4]=(a_{34}^1+c_{43}^1)\,X_1+c_{43}^2\,X_2+(b_{34}^1-b_{43}^1)\,JX_1-b_{43}^2\,JX_2,
 \end{array}$$
where $\frg$ is 4-step nilpotent if $c_{44}^3 = 0$ and 5-step nilpotent otherwise.  Imposing the vanishing of
$Jac(X_4, JX_4, JX_3)$, a contradiction is obtained.

\smallskip\noindent
$\triangleright$ \ Let us consider $\mathcal V_3\subset\frg_3$.  In this case, the arrangement of the ascending central series starts as follows:
$$\frg_1=\langle X_1\rangle, \qquad \frg_2=\langle X_1,X_2,X_3,JX_2\rangle, \qquad
  \frg_3=\langle X_1,X_2,X_3, JX_2, JX_3\rangle.$$
 The unknown brackets in~\eqref{bracket_dim45} can be determined as follows.
 Due to the Nijenhuis condition and the nilpotency of the algebra, one has:
  $$\begin{array}{l}
 [X_4,JX_3]=b_{34}^1\,X_1+b_{43}^2\,X_2+c_{43}^2\,JX_2,\\[2pt]
 [JX_3,JX_4]=a_{34}^1\,X_1+c_{43}^2\,X_2-b_{43}^2\,JX_2.
 \end{array}$$
 Furthermore, for both $\mathcal V_4\cap \frg_4\neq \{0\}$ and $\mathcal V_4\cap \frg_4=\{0\}$, it is possible to set
  $$\begin{array}{l}
 [X_4,JX_1]=b_{41}^2\,X_2+b_{41}^3\,X_3 + c_{41}^2\,JX_2+c_{41}^3\,JX_3,\\[2pt]
 [X_4,JX_4]=b_{44}^1\,X_1+b_{44}^2\,X_2+b_{44}^3\,X_3+c_{44}^1\,JX_1+c_{44}^2\,JX_2 +c_{44}^3\,JX_3, \\[2pt]
 [JX_1,JX_4]=c_{41}^2\,X_2+c_{41}^3\,X_3 - b_{41}^2\,JX_2-b_{41}^3\,JX_3,
 \end{array}$$
where $\frg$ is 4-step nilpotent if $c_{44}^1 = 0$ and 5-step nilpotent otherwise.
The vanishing of $Jac(X_4, JX_4, JX_1)$ generates a new contradiction.
\end{addmargin}
\end{addmargin}

Observe that we have rejected the case in which $\frg_3$ is not $J$-invariant.
Hence, we are led to consider the opposite situation.

\smallskip\noindent
$\bullet$ \ \textsc{Study when} $\frg_3$ \textsc{is} $J$\textsc{-invariant}.

\smallskip\noindent
One necessarily has
$\frg_3\supseteq \mathcal V_1\oplus \mathcal V_2\oplus \mathcal V_3$ and
\begin{equation}\label{bracket_caso11211}
\begin{array}{l}
[X_4,JX_1]=b_{41}^2\,X_2+c_{41}^2\,JX_2,\\[2pt]
[X_4,JX_3]=b_{34}^1\,X_1+b_{43}^2\,X_2+c_{43}^2\,JX_2, \\[2pt]
[X_4,JX_4]=b_{44}^1\,X_1+b_{44}^2\,X_2+b_{44}^3\,X_3+c_{44}^1\,JX_1+c_{44}^2\,JX_2+c_{44}^3\,JX_3, \\[2pt]
[JX_1,JX_4]=c_{41}^2\,X_2-b_{41}^2\,JX_2, \\[2pt]
 [JX_3,JX_4]=a_{34}^1\,X_1+c_{43}^2\,X_2-b_{43}^2\,JX_2,
\end{array}
\end{equation}
where $\frg$ is 4-step nilpotent if $c_{44}^1 = c_{44}^3=0$ and 5-step nilpotent otherwise.  From $Jac(X_2, JX_3, Y)$ and $Jac(X_3, JX_1, Y)$ for $Y=X_4, JX_4$, one obtains $a_{23}^1 = b_{22}^1 = b_{23}^1=0$ in order to preserve the desired arrangement of the series.  Furthermore, it is possible to give particular solutions to the remaining Jacobi identities, showing that our two cases are valid.
\end{proof}

\subsubsection{Study for dim\,$\frg_2 = 5$}
\noindent
We finally focus our attention on dim\,$\frg_2=5$. In this case it is possible to find $X_1, X_2, X_3\in\frg$ such that:
$$\frg_{1}=\langle X_{1}\rangle,\quad  \frg_{2}=\langle X_{1},X_{2},X_{3},JX_{2}, JX_3\rangle.$$
Let $\{X_i, JX_i\}_{i=1}^3$ belong to the $J$-adapted basis $\mathcal B$ and consider $X_4, JX_4$
generators of the $J$-invariant subspace $\mathcal V_4 =\frg\setminus (\mathcal V_1\oplus \mathcal V_2\oplus \mathcal V_3)$.
The Lie brackets of $\frg$ are given by:
\begin{equation}\label{bracket_dim5}
\begin{array}{lll}
[X_1,\cdot\  ]=0,
  & [X_{3},JX_{2}]=b_{23}^{1}\,X_{1},
  & [JX_1,JX_k]_{k=2,3}=0, \\[2pt]
[X_2,X_k]_{k=3,4}=a_{2k}^1\,X_1,
  & [X_{3},JX_{k}]_{k=3,4} = b_{3k}^1 X_1,
  & [JX_1,JX_4] \text{ unknown}, \\[2pt]
[X_2,JX_1]=0,
  & [X_4, JX_1] \text{ unknown},
  & [JX_{2},JX_{k}]_{k=3,4} =a_{2k}^{1}\,X_{1}, \\[2pt]
[X_2,JX_k]_{k=2,3,4}=b_{2k}^1\,X_1,
  & [X_{4},JX_{2}]=b_{24}^{1}\,X_{1},
  & [JX_{3},JX_{4}]=a_{34}^1 X_1,\\[2pt]
[X_3, X_4]=a_{34}^{1}\,X_{1},
  & [X_{4},JX_{3}]= b_{34}^1X_1,
  &\\[2pt]
[X_3, JX_1] =0,
  & [X_4,JX_4] \text{ unknown},
  &
\end{array}
\end{equation}
where all the coefficients are real numbers.

\begin{proposition}\label{serie-dimg25}
Let $(\frg,J)$ be an $8$-dimensional NLA endowed with an SnN complex structure. If
$\dim \frg_2=5$, then the ascending central series of $\frg$ satisfies:
$$\text{dim }\frg_1 = 1,\quad \frg_2 = \frg_1\oplus\mathcal V_2\oplus\mathcal V_3,\quad \frg_3 \text{ is } J\text{-invariant},$$ where $\mathcal V_i$ are $2$-dimensional $J$-invariant spaces.
\end{proposition}

\begin{proof}
We observe that $\frg = \frg_2\oplus \mathcal V_4\oplus J\frg_1$, so we are in the conditions of Lemma~\ref{lemma2_serie_asc} with $k=2$.  As a consequence, we can directly  conclude that either $\frg_3 = \frg$, if $\mathcal V_4\cap \frg_3\neq \{0\}$,
or $\frg_4=\frg$, if $\mathcal V_4\cap \frg_3=\{0\}$.  In the second case, we notice that $\mathcal V_1\subset \frg_3$, which means $\frg_3 = \mathcal V_1\oplus\mathcal V_2\oplus\mathcal V_3$. Hence, in both cases $\frg_3$ is $J$-invariant.
To finish the proof, it suffices to show that the two situations are valid.

Let us first assume that $\mathcal V_4\cap \frg_3\neq \{0\}$.  Then,
$$\frg_{1}=\langle X_{1}\rangle,\quad \frg_{2}=\langle X_{1},X_{2}, X_3, JX_{2}, JX_3\rangle,\quad \frg_3=\frg,$$
and we can set all the unknown real brackets in~\eqref{bracket_dim5}:
\begin{equation}\label{bracket_caso1111}
\begin{array}{l}
[X_4,JX_1]=b_{41}^2\,X_2+b_{41}^3\,X_3+c_{41}^2\,JX_2+c_{41}^3\,JX_3,\\[2pt]
[X_4,JX_4]=b_{44}^1\,X_1+b_{44}^2\,X_2+b_{44}^3\,X_3+c_{44}^2\,JX_2+c_{44}^3\,JX_3,\\[2pt]
[JX_1,JX_4]=c_{41}^2\,X_2+c_{41}^3\,X_3-b_{41}^2\,JX_2-b_{41}^3\,JX_3.
\end{array}
\end{equation}
The Jacobi identity gives rise to a quadratic system with 9 equations and 17 (real) unknowns that
admit appropriate solutions.  Hence, this case leads to a first positive result.

We now study the opposite situation, namely, $\mathcal V_4\cap \frg_3=\{0\}$. Then,
$$\frg_1=\langle X_1\rangle, \quad \frg_2=\langle X_1,X_2,X_3,JX_2,JX_3\rangle,\quad \frg_3=\langle X_1,X_2,X_3,JX_1,JX_2,JX_3\rangle, \quad \frg_4=\frg,$$ and we have:
\begin{equation*}
\begin{array}{l}
[X_4,JX_1]=b_{41}^2\,X_2+b_{41}^3\,X_3+c_{41}^2\,JX_2+c_{41}^3\,JX_3,\\[2pt]
[X_4,JX_4]=b_{44}^1\,X_1+b_{44}^2\,X_2+b_{44}^3\,X_3+c_{44}^1\,JX_1+c_{44}^2\,JX_2+c_{44}^3\,JX_3,\\[2pt]
[JX_1,JX_4]=c_{41}^2\,X_2+c_{41}^3\,X_3-b_{41}^2\,JX_2-b_{41}^3\,JX_3.
\end{array}
\end{equation*}
In particular, notice that $c_{44}^1\neq 0$ or otherwise, we would go back to the first case.
Also here it is possible to find a solution to the Jacobi identities preserving the arrangement of the ascending central series.
\end{proof}


As a consequence of the previous results, we have found every $8$-dimensional NLA $\frg$ admitting an SnN
complex structure $J$. Indeed, we have provided a $J$-adapted basis of each $\frg$ additionally adapted
to the structure of the ascending central series $\{\frg_k\}_k$. This fact motivates the following definition.

\begin{definition}\label{def-doubly-adapted}
Let $\frg$ be an $s$-step nilpotent Lie algebra of dimension $2n$ endowed with an almost-complex structure $J$.
A $J$-adapted basis $\mathcal B=\{X_k,JX_k\}_{k=1}^n$ of $\frg$ will be called \emph{doubly adapted}
if there is a permutation $\mathcal B^{\sigma}=\{V_{1},\ldots,V_{2n}\}$ of the elements of $\mathcal B$
such that $\mathcal B^{\sigma}$ is
adapted to the ascending central series $\{\frg_{k}\}_{k}$, i.e.,
$$\frg_{1}=\langle V_{1},\ldots,V_{m_{1}}\rangle\subset
    \frg_{2}=\langle V_{1},\ldots,V_{m_{1}},V_{m_{1}+1},\ldots,V_{m_{2}}\rangle\subset\cdots\subset
    \frg_{s}=\frg=\langle V_{1},\ldots,V_{2n}\rangle,$$
being $m_{k}=\dim \frg_{k}$.
\end{definition}

\subsection{Complex structure equations}\label{SnN-equations-dim8}
In the previous section, we have completely determined those NLAs $\frg$ of dimension eight
admitting an SnN complex structure $J$. Here, we describe the complex structure equations
for every such pair $(\frg,J)$.

Let $\frg$ be an NLA of dimension $8$ endowed with a strongly non-nilpotent complex structure.
Let $\mathcal B=\{X_{1}, X_2, X_3, X_4, JX_1, JX_2, JX_3, JX_4 \}$ be the doubly adapted
basis that we have constructed in each case of Section~\ref{sec:g34}, in this specific order.
Consider its dual basis $\mathcal B^*=\{e^1,\ldots, e^8\}$.
We recall that the differential of any element $e\in\frg^*$ can be computed using the Lie brackets of $\frg$
by means of the well-known formula
$$de(X,Y)=-e([X,Y]),$$
where $X,Y\in\frg$.
Since $J$ is a complex structure, one can then construct the $(1,0)$-basis
\begin{equation}\label{base-eta}
\omega^1=e^4-ie^8,\quad  \omega^2=e^3-ie^7,\quad \omega^3=e^2-ie^6,\quad \omega^4=e^1-ie^5,
\end{equation}
and calculate the (complex) structure equations of $(\frg,J)$.
We present them in different propositions according to the dimension of $\frg_2$.


\begin{proposition}\label{complexification_(1,3,5,8)+(1,3,5,6,8)}
Let $J$ be a strongly non-nilpotent complex structure on an 8-dimensional nilpotent
Lie algebra $\frg$ such that $\dim \frg_1=1$ and $\dim \frg_2=3$.
The pair $(\frg,J)$ is parametrized by
\begin{equation*}
\left\{\begin{array}{rcl}
d\omega^1 &=& 0,\\
d\omega^2 &=& A\,\omega^{1\bar 1} -B(\omega^{14}- \omega^{1\bar 4}),\\
d\omega^3 &=& (C-D)\,\omega^{12}-E\,(\omega^{14}- \omega^{1\bar 4})+F\,\omega^{1\bar 1}+(G+D)\,\omega^{1\bar 2}
  -H\,(\omega^{24}-\omega^{2\bar 4})\\
                                 &&  +(C-G)\,\omega^{2\bar1}+K\,\omega^{2\bar 2},\\
d\omega^4 &=& L\,\omega^{1\bar 1}+M\,\omega^{1\bar 2}+N\,\omega^{1\bar 3}-\bar{M}\,\omega^{2\bar 1}
  +i\,s\,\omega^{2\bar 2}+P\,\omega^{2\bar 3}-\bar{N}\,\omega^{3\bar 1}-\bar{P}\,\omega^{3\bar2},\end{array}\right.
\end{equation*}
where the coefficients are complex numbers with $s\in\mathbb R$ and satisfying the Jacobi identity. Moreover:
\begin{itemize}
\item[\textrm{(i)}] if $(\dim \frg_k)_{k}=(1,3,8)$ then $A=B=0$ and $\Real L=0$,
\item[\textrm{(ii)}] if $(\dim \frg_k)_{k}=(1,3,6,8)$ then $B=H=K=P=0$,
\item[\textrm{(iii)}] if $(\dim \frg_k)_{k}=(1,3,5,8)$ then $K=P=0$ and $\Real L=0$,
\item[\textrm{(iv)}] if $(\dim \frg_k)_{k}=(1,3,5,6,8)$ then $H=K=P=s=0$ and $\Real L\neq 0$.
\end{itemize}
\end{proposition}

\begin{proof}
On the one hand, from the Lie brackets obtained in the proof of Lemma \ref{lema2}
one has real structure equations:
$$\left\{\begin{array}{rcl}
de^1 &=& -a_{23}^1\,e^{23}-a_{24}^1\,e^{24}-b_{23}^1\,e^{27}-b_{24}^1\,e^{28}
                -a_{34}^1\,e^{34}-b_{23}^1\,e^{36}-b_{33}^1\,e^{37} \\
          && -b_{34}^1\,e^{38}-b_{24}^1\,e^{46}-b_{34}^1\,e^{47}-b_{44}^1\,e^{48}
                -a_{23}^1\,e^{67}-a_{24}^1\,e^{68}-a_{34}^1\,e^{78}, \\[5pt]
de^2 &=& -a_{34}^2\,e^{34}-b_{31}^2\,e^{35}-b_{33}^2\,e^{37}-b_{34}^2\,e^{38}
                -b_{41}^2\,e^{45}-b_{43}^2\,e^{47} \\
           && -b_{44}^2\,e^{48}-c_{31}^2\,e^{57}-c_{41}^2\,e^{58}-(a_{34}^2-c_{34}^2+c_{43}^2)\,e^{78},\\[5pt]
de^3 &=& -b_{44}^3\,e^{48}, \\[5pt]
de^4 &=& 0, \\[5pt]
de^5 &=& -c_{44}^1\,e^{48}, \\[5pt]
de^6 &=& -\alpha_{34}^2\,e^{34}-c_{31}^2\,e^{35}-c_{33}^2\,e^{37}-c_{34}^2\,e^{38}
                -c_{41}^2\,e^{45}-c_{43}^2\,e^{47} \\
           && -c_{44}^2\,e^{48}+b_{31}^2\,e^{57}+b_{41}^2\,e^{58}-(\alpha_{34}^2+b_{34}^2-b_{43}^2)\,e^{78},\\[5pt]
de^7 &=& -c_{44}^3\,e^{48}, \\[5pt]
de^8 &=& 0,
\end{array}\right.$$
where the structure constants satisfy $b_{44}^3=c_{44}^1=c_{44}^3=0$ for $(\dim \frg_k)_{k}=(1,3,8)$ and
$a_{23}^1=b_{23}^1=b_{31}^2=b_{33}^2=c_{31}^2=c_{33}^2=0$ for $(\dim \frg_k)_{k}=(1,3,6,8)$.

On the other hand, using the Lie brackets found in the proof of Lemma \ref{lema3} we get:
$$\left\{\begin{array}{rcl}
de^1 &=& -a_{24}^1\,e^{24}-b_{24}^1\,e^{28}
                -a_{34}^1\,e^{34}-b_{33}^1\,e^{37}-b_{34}^1\,e^{38} \\
          && -b_{24}^1\,e^{46}-b_{34}^1\,e^{47}-b_{44}^1\,e^{48}
                -a_{24}^1\,e^{68}-a_{34}^1\,e^{78}, \\[5pt]
de^2 &=& -a_{34}^2\,e^{34}-b_{31}^2\,e^{35}-b_{34}^2\,e^{38}
                -b_{41}^2\,e^{45}-b_{43}^2\,e^{47} \\
           && -b_{44}^2\,e^{48}-c_{31}^2\,e^{57}-c_{41}^2\,e^{58}-(a_{34}^2-c_{34}^2+c_{43}^2)\,e^{78},\\[5pt]
de^3 &=& -b_{41}^3\,e^{45}-b_{44}^3\,e^{48}-c_{41}^3\,e^{58}, \\[5pt]
de^4 &=& 0, \\[5pt]
de^5 &=& -c_{44}^1\,e^{48}, \\[5pt]
de^6 &=& -\alpha_{34}^2\,e^{34}-c_{31}^2\,e^{35}-c_{34}^2\,e^{38}
                -c_{41}^2\,e^{45}-c_{43}^2\,e^{47} \\
           && -c_{44}^2\,e^{48}+b_{31}^2\,e^{57}+b_{41}^2\,e^{58}-(\alpha_{34}^2+b_{34}^2-b_{43}^2)\,e^{78},\\[5pt]
de^7 &=& -c_{41}^3\,e^{45}-c_{44}^3\,e^{48}+b_{41}^3\,e^{58}, \\[5pt]
de^8 &=& 0,
\end{array}\right.$$
where the structure constants satisfy $c_{44}^1=0$ for $(\dim \frg_k)_{k}=(1,3,5,8)$, and
$b_{31}^2=c_{31}^2=b_{33}^1=0$ for $(\dim \frg_k)_{k}=(1,3,5,6,8)$.

Constructing the $(1,0)$-basis $\{\omega^k\}_{k=1}^4$ described in \eqref{base-eta} we get the desired result,
simply defining the following complex numbers:
\begin{equation}\label{letras}
\begin{array}{lllll}
A=\dfrac{c_{44}^3+i\,b_{44}^3}{2},
  & D=\dfrac{c_{43}^2+i\,b_{43}^2}{2},
  & G=\dfrac{a_{34}^2-i\,\alpha_{34}^2}{2},
  & L=\dfrac{c_{44}^1+i\,b_{44}^1}{2},
  & P=\dfrac{a_{23}^1+i\,b_{23}^1}{2}, \\[7pt]
B=\dfrac{c_{41}^3+i\,b_{41}^3}{2},
  & E=\dfrac{c_{41}^2+i\,b_{41}^2}{2},
  & H=\dfrac{c_{31}^2+i\,b_{31}^2}{2},
  & M=\dfrac{a_{34}^1+i\,b_{34}^1}{2},
  & s = \dfrac{b_{33}^1}{2}.\\[7pt]
C=\dfrac{c_{34}^2+i\,b_{34}^2}{2},
  & F=\dfrac{c_{44}^2+i\,b_{44}^2}{2},
  & K=\dfrac{c_{33}^2+i\,b_{33}^2}{2},
  & N=\dfrac{a_{24}^1+i\,b_{24}^1}{2},
  &
\end{array}
\end{equation}
\end{proof}


\begin{proposition}\label{complexification_(1,4,8)+(1,4,6,8)}
Let $J$ be a strongly non-nilpotent complex structure on an 8-dimensional nilpotent
Lie algebra $\frg$ such that $\dim \frg_1=1$ and $\dim \frg_2=4$.
The pair $(\frg,J)$ is parametrized by the structure equations
\begin{equation*}
\left\{\begin{array}{rcl}
d\omega^1 &=& 0,\\
d\omega^2 &=& A\,\omega^{1\bar 1},\\
d\omega^3 &=& -D (\omega^{12}-\omega^{1\bar 2})-E\,(\omega^{14}-\omega^{1\bar 4})+F\,\omega^{1\bar 1},\\
d\omega^4 &=& L\,\omega^{1\bar 1}+M\,\omega^{1\bar 2}+N\,\omega^{1\bar 3}-\bar{M}\,\omega^{2\bar 1}
  +i\,s\,\omega^{2\bar 2}-\bar{N}\,\omega^{3\bar 1},
\end{array}\right.\end{equation*}
where the coefficients are complex numbers with $s\in\mathbb R$, and satisfying the Jacobi identity. Furthermore:
\begin{itemize}
\item[\textrm{(i)}] if $(\dim \frg_k)_{k}=(1,4,8)$ then $(\Real A,\Real L)=(0,0)$,
\item[\textrm{(ii)}] if $(\dim \frg_k)_{k}=(1,4,6,8)$ then $(\Real A,\Real L)$ $\neq(0,0)$.
\end{itemize}
\end{proposition}

\begin{proof}
Using the Lie brackets~\eqref{bracket_dim45} and \eqref{bracket_caso11211} we obtain the  real structure equations:
$$\left\{\begin{array}{rcl}
de^1 &=& -a_{24}^1\,e^{24}-b_{24}^1\,e^{28}
                -a_{34}^1\,e^{34}-b_{33}^1\,e^{37}-b_{34}^1\,e^{38}
           -b_{24}^1\,e^{46}-b_{34}^1\,e^{47}-b_{44}^1\,e^{48}
                -a_{24}^1\,e^{68}-a_{34}^1\,e^{78}, \\[5pt]
de^2 &=& -b_{41}^2\,e^{45}-b_{43}^2\,e^{47}-b_{44}^2\,e^{48}-c_{41}^2\,e^{58}
                -c_{43}^2\,e^{78},\\[5pt]
de^3 &=& -b_{44}^3\,e^{48}, \\[5pt]
de^4 &=& 0, \\[5pt]
de^5 &=& -c_{44}^1\,e^{48}, \\[5pt]
de^6 &=& -c_{41}^2\,e^{45}-c_{43}^2\,e^{47} -c_{44}^2\,e^{48}+b_{41}^2\,e^{58}+b_{43}^2\,e^{78},\\[5pt]
de^7 &=& -c_{44}^3\,e^{48}, \\[5pt]
de^8 &=& 0.
\end{array}\right.$$
Now, construct the $(1,0)$-basis $\{\omega^k\}_{k=1}^4$ described in \eqref{base-eta} to get
the result, where the parameters follow~\eqref{letras}.
\end{proof}


\begin{proposition}\label{complexification_(1,5,8)+(1,5,6,8)}
Let $J$ be an SnN complex structure on an 8-dimensional NLA $\frg$
such that $\dim \frg_1=1$ and $\dim \frg_2=5$.
The pair $(\frg,J)$ is parametrized by the structure equations
\begin{equation*}
\left\{\begin{array}{l}
d\omega^1 \ \, = \ \, 0,\\
d\omega^2 \ \, = \ \, A\,\omega^{1\bar 1} -B(\omega^{14}- \omega^{1\bar 4}),\\
d\omega^3 \ \, = \ \, F\,\omega^{1\bar 1} -E(\omega^{14}- \omega^{1\bar 4}),\\
d\omega^4 \ \, = \ \, L\,\omega^{1\bar 1}+M\,\omega^{1\bar 2}+N\,\omega^{1\bar 3}-\bar{M}\,\omega^{2\bar 1}
  +i\,s\,\omega^{2\bar 2}+P\,\omega^{2\bar 3}
  -\bar{N}\,\omega^{3\bar 1}-\bar P\,\omega^{3\bar 2}+i\,t\,\omega^{3\bar 3},
\end{array}\right.
\end{equation*}
where the coefficients are complex numbers, $s, t\in\mathbb R$, and they satisfy the Jacobi identity. Moreover:
\begin{itemize}
\item[\textrm{(i)}] if $(\dim \frg_k)_{k}=(1,5,8)$ then $\Real L=0$,
\item[\textrm{(ii)}] if $(\dim \frg_k)_{k}=(1,5,6,8)$ then $\Real L\neq 0$.
\end{itemize}
\end{proposition}

\begin{proof}
Directly from the Lie brackets~\eqref{bracket_dim5} and~\eqref{bracket_caso1111} we obtain the  real structure equations:
$$\left\{\begin{array}{rcl}
de^1 &=& -a_{23}^1\,e^{23}-a_{24}^1\,e^{24}-b_{22}^1\,e^{26}-b_{23}^1\,e^{27}-b_{24}^1\,e^{28}
                -a_{34}^1\,e^{34}-b_{23}^1\,e^{36}-b_{33}^1\,e^{37}, \\
          && -b_{34}^1\,e^{38}-b_{24}^1\,e^{46}-b_{34}^1\,e^{47}-b_{44}^1\,e^{48}
                -a_{23}^1\,e^{67}-a_{24}^1\,e^{68}-a_{34}^1\,e^{78}, \\[5pt]
de^2 &=& -b_{41}^2\,e^{45}-b_{44}^2\,e^{48}-c_{41}^2\,e^{58},\\[5pt]
de^3 &=& -b_{41}^3\,e^{45}-b_{44}^3\,e^{48}-c_{41}^3\,e^{58}, \\[5pt]
de^4 &=& 0, \\[5pt]
de^5 &=& -c_{44}^1\,e^{48}, \\[5pt]
de^6 &=& -c_{41}^2\,e^{45}-c_{44}^2\,e^{48}+b_{41}^2\,e^{58},\\[5pt]
de^7 &=& -c_{41}^3\,e^{45}-c_{44}^3\,e^{48}+b_{41}^3\,e^{58}, \\[5pt]
de^8 &=& 0,
\end{array}\right.$$
where $c_{44}^1=0$ for $(\dim \frg_k)_{k}=(1,5,8)$ and
$c_{44}^1\neq 0$ for $(\dim \frg_k)_{k}=(1,4,6,8)$.
Constructing the $(1,0)$-basis $\{\omega^k\}_{k=1}^4$ given by \eqref{base-eta}, we get
our complex equations above,
where $t=\frac{b_{22}^1}{2}$ and the rest of the parameters follow~\eqref{letras}.
\end{proof}


\section*{Acknowledgments}
\noindent
Some results of this paper overlap with the first author's Ph. D. thesis at the
University of Zaragoza, under the supervision of the second and third authors; we would like to thank
D. Angella, A. Fino, S. Ivanov and M.T. Lozano for their interest and useful comments
and suggestions on the subject.
The final version of this paper was prepared during the stay of the first and the last authors at the
Fields Institute. They would like to thank this institution for its kind hospitality and support.
This work has been partially supported by the projects MINECO (Spain) MTM2014-58616-P, and
Gobierno de Arag\'on/Fondo Social Europeo--Grupo Consolidado E15 Geometr\'{\i}a.


\end{document}